\documentclass[11pt,english]{amsart}
\usepackage{amsfonts, amssymb, amsmath, amsthm,eucal,latexsym,nicefrac,mathrsfs}
\usepackage{dsfont, mathrsfs,enumerate,eucal,fontenc}
\usepackage[varg]{txfonts}
\usepackage{xspace}

\usepackage[english]{babel}

\usepackage{pstricks}
\usepackage{color,pstcol}
\usepackage{graphics} 
\usepackage{psfrag}  
\usepackage[all]{xy}

\usepackage{xcolor}
\definecolor{rouge}{rgb}{0.85,0.1,.4}
\definecolor{bleu}{rgb}{0.1,0.2,0.9}
\definecolor{violet}{rgb}{0.7,0,0.8}
\usepackage[colorlinks=true,linkcolor=bleu,urlcolor=violet,citecolor=rouge]{hyperref}

\DeclareMathAlphabet{\mathpzc}{OT1}{pzc}{m}{it}

\theoremstyle{plain}
\newtheorem{theorem}{Theorem}[section]
\newtheorem{lemma}[theorem]{Lemma}
\newtheorem{theo}[theorem]{Theorem}

\newtheorem{coro}[theorem]{Corollary}
\newtheorem{prop}[theorem]{Proposition}

\theoremstyle{definition}

\theoremstyle{remark}
\newtheorem{rema}[theorem]{Remark}

\def\k{{\Bbbk}}

\def\rg{\ell}               

\def\poie#1#2#3#4#5#6#7#8#9{\def\un{#5#6#7#8#9}\def\deux{#6#7#8#9}\def\trois{#2#4#8#9}
\def\quatre{#8#9}\def\cinq{#5#6#7}\def\six{#6#7}\def\sept{#2#4}
\ifx\un\empty {#1}_{#2}{#3 \hskip 0.15em}{#1}_{#4} \else \ifx\deux\empty 
{#5}(#1_{#2}){#3 \hskip 0.15em}{#5}(#1_{#4})
\else \ifx\trois\empty {#5}_{#6}(#1){#3 \hskip 0.15em}{#5}_{#7}(#1) 
\else \ifx\quatre\empty {#5}_{#6}(#1_#2){#3 \hskip 0.15em}{#5}_{#7}(#1_#4) 
\else \ifx\cinq\empty {#1}_{#2}^{#8}{#3 \hskip 0.15em}#1_#4^{#9} 
\else \ifx\six\empty {#5}(#1_{#2}^{#8}){#3 \hskip 0.15em}{#5}(#1_{#4}^{#9}) 
\else \ifx\sept\empty {#5}_{#6}(#1)^{#8}{#3 \hskip 0.15em}{#5}_{#7}(#1)^{#9} \else
{#5}_{#6}(#1_{#2}^{#8})^{#9}{#3 \hskip 0.15em}{#5}_{#7}(#1_{#4}^{#8})^{#9} 
\fi \fi \fi \fi \fi \fi \fi}
\def\poi#1#2#3#4#5#6#7{\def\un{#5#6#7}\def\deux{#6#7}
\def\trois{#2#4} \def\cinq{#3#4#5}
\ifx\un\empty {#1}_{#2}{#3 \hskip 0.15em}{#1}_{#4} \else
\ifx\deux\empty {#5}(#1_{#2}){#3 \hskip 0.15em}{#5}(#1_{#4}) \else
\ifx\trois\empty {#5}_{#6}(#1){#3 \hskip 0.15em}{#5}_{#7}(#1) \else
{#5_{#6}}(#1_{#2}){#3 \hskip 0.15em}{#5_{#7}}(#1_{#4}) \fi \fi \fi}

\def\dv#1#2{\langle {#1},{#2}\rangle}

\def\ec#1#2#3#4#5{\def\un{#3#4#5}\def\deux{#3#5}\def\trois{#3}
\def\four{#2#4#5}\def\five{#2#5}\def\six{#2}\def\seven{#3#4}
\def\eight{#2#4} \def\nine{#2#3#4}
\ifx\nine\empty {\rm #1}_{#5} \else
\ifx\un\empty {\rm #1}({\goth #2}) \else
\ifx\deux\empty {\rm #1}({\goth #2}_{#4}) \else
\ifx\trois\empty {\rm #1}_{#5}({\goth #2}_{#4}) \else
\ifx\four\empty {\rm #1}(#3) \else
\ifx\five\empty {\rm #1}(#3_{#4}) \else
\ifx\six\empty {\rm #1}_{#5}(#3_{#4}) \else
\ifx\seven\empty {\rm #1}_{#5} ({\goth#2})\else
\ifx\eight\empty {\rm #1}_{#5}({#3})
\fi \fi \fi \fi \fi \fi \fi \fi \fi}
\def\hec#1#2#3#4#5{\def\un{#3#4#5}\def\deux{#3#5}\def\trois{#3}
\def\four{#2#4#5}\def\five{#2#5}\def\six{#2}\def\seven{#3#4}
\def\eight{#2#4} \def\nine{#2#3#4}
\ifx\nine\empty \hat{{\rm #1}}_{#5} \else
\ifx\un\empty \hat{{\rm #1}}({\goth #2}) \else
\ifx\deux\empty \hat{{\rm #1}}({\goth #2}_{#4}) \else
\ifx\trois\empty \hat{{\rm #1}}_{#5}({\goth #2}_{#4}) \else
\ifx\four\empty \hat{{\rm #1}}(#3) \else
\ifx\five\empty \hat{{\rm #1}}(#3_{#4}) \else
\ifx\six\empty \hat{{\rm #1}}_{#5}(#3_{#4}) \else
\ifx\seven\empty \hat{{\rm #1}}_{#5} ({\goth#2})  \else
\ifx\eight\empty \hat{{\rm #1}}_{#5}({#3})
\fi \fi \fi \fi \fi \fi \fi \fi \fi}

\def\es#1#2{\ec {#1}{}{#2}{}{}}

\def\ai#1#2#3{\def\deux{#2#3} \def\trois{#3} \def\quatre{#2} 
\ifx\deux\empty \es S{{\goth #1}}^{{\goth #1}} \else
\ifx\trois\empty \es S{{\goth #1}^{#2}}^{{\goth #1}^{#2}} \else
\ifx\quatre\empty \es S{{\goth #1}_{#3}}^{{\goth #1}_{#3}} \else
\es S{{\goth #1}_{#3}^{#2}}^{{\goth #1}_{#3}^{#2}} \fi \fi \fi}

\def\Bbb{\mathbb}
\def\goth{\mathfrak}
\def\cal{\mathcal}

\def\gi#1#2#3#4{\def\trois{#3#4} \def\quatre{#4}\def\cinq{#3}\ifx\trois\empty {\rm i}_{#1,{\goth #2}}
\else \ifx\quatre\empty {\rm i}_{#1_{#3},{\goth #2}} \else\ifx\cinq\empty {\rm i}_{#1,{\goth #2}_{#4}} \else {\rm i}_{#1_{#3},{\goth #2}_{#4}} \fi \fi \fi}
\def\j#1#2{\def\deux{#2} \ifx\deux\empty {\rm rk}\hskip .125em{{\goth #1}} \else {\rm rk}\hskip .125em{{\goth #1}_{#2}} \fi}
\def\aj#1#2{\def\deux{#2} \ifx\deux\empty {\rm j}_{{\goth #1}} \else {\rm j}_{{\goth #1}_{#2}} \fi}
\def\an#1#2{\def\deux{#2} \ifx\deux\empty {\cal O}_{#1} \else {\cal O}_{#1,#2} \fi }
\def\han#1#2{\def\deux{#2} \ifx\deux\empty {\hat{{\cal O}}}_{#1} \else {\hat{{\cal O}}}_{#1,#2} \fi }

\def\dim{{\rm dim}\hskip .125em}

\def\ad{{\rm ad}\hskip .1em}
\def\Ad{{\rm Ad}\hskip .1em}   

\def\det{{\rm det}\hskip .125em}
\def\j#1#2{\ell _{{\goth #1}_{#2}}}

\def\n{{\rm n}}

\def\u{{\rm u}}

\def\ex #1#2{\mbox{$\bigwedge^{#1}(#2)$}} 

\def\b#1#2{{\mathrm {b}}_{{\mathfrak{#1}}_{#2}}}

\title
[Subspaces of the exterior algebra]
{On some subspaces of the exterior algebra of a simple Lie algebra}

\author
[J-Y Charbonnel]{Jean-Yves Charbonnel}

\address{Jean-Yves Charbonnel, Universit\'e de Paris - CNRS \\
Institut de Math\'ematiques de Jussieu - Paris Rive Gauche\\
UMR 7586 \\ Groupes, repr\'esentations et g\'eom\'etrie \\
B\^atiment Sophie Germain \\ Case 7012 \\ 
75205 Paris Cedex 13, France}
\email{jean-yves.charbonnel@imj-prg.fr}



\date\today

\begin{document}

\large

\begin{abstract}
  In this article, we are interested in some subspaces of the exterior algebra of a simple
  Lie algebra ${\goth g}$. In particular, we prove that some graded subspaces of degree
  $d$ generate the ${\goth g}$-module $\ex d{{\goth g}}$ for some integers $d$.
\end{abstract}

\maketitle

\setcounter{tocdepth}{1}
\tableofcontents

\section{Introduction} \label{int}
In this note, the base field $\k$ is algebraically closed of characteristic $0$, 
${\goth g}$ is a simple Lie algebra of finite dimension, $\rg$ is its rank,
and $G$ is its adjoint group.    

\subsection{Main results} \label{int1}
Let ${\goth b}$ be a Borel subalgebra of ${\goth g}$, ${\goth h}$ a Cartan subalgebra
of ${\goth g}$, contained in ${\goth b}$ and ${\goth u}$ the nilpotent radical of 
${\goth b}$. Set $\b g{} := \dim {\goth b}$ and $n := \b g{}-\rg$. For $k$ a nonnegative 
integer, let $\ex k{{\goth g}}$ be the component of degree $k$ of the exterior algebra 
$\ex {}{{\goth g}}$ of ${\goth g}$. The adjoint action of $G$ in ${\goth g}$ induces an 
action of $G$ in $\ex {}{{\goth g}}$. For all $k$, $\ex k{{\goth g}}$ is invariant under 
this action.

For ${\goth p}$ parabolic subalgebra of ${\goth g}$, containing ${\goth b}$, denote
by ${\goth p}_{\u}$ the nilpotent radical of ${\goth p}$, ${\goth l}$ the reductive
factor of ${\goth p}$, containing ${\goth h}$, ${\goth z}$ the center of ${\goth l}$
and ${\goth p}_{-,\u}$ the complement to ${\goth p}$ in ${\goth g}$, invariant under 
the adjoint action of ${\goth h}$. Let ${\goth d}$ be the derived algebra of ${\goth l}$
and $\poi {{\goth d}}1{,\ldots,}{\n}{}{}{}$ its simple factors. Set:
$$ E_{{\goth p}} := {\goth p}_{-,\u}\oplus {\goth z}\oplus {\goth p}_{\u}, \quad 
n_{i} := \dim {\goth d}_{i}\cap {\goth u}, $$ $$  
{\Bbb I}_{k} := \{(\poi j1{,\ldots,}{\n}{}{}{}) \in {\Bbb N}^{\n} \; \vert \;
j_{1}\leq n_{1} ,\ldots, j_{\n}\leq n_{\n},\poi j1{+\cdots +}{\n}{}{}{} = k\}$$ 
for $k$ positive integer. Denote by $V'_{k,{\goth p}}$ and 
$V_{k,{\goth p}}$ the subspaces of $\ex k{{\goth g}}$,
$$ V'_{k,{\goth p}} := \bigoplus _{(\poi j1{,\ldots,}{\n}{}{}{}) \in {\Bbb I}_{k}}
\ex {j_{1}}{{\goth d}_{1}}\wedge \cdots \wedge \ex {j_{\n}}{{\goth d}_{\n}} \quad  
\text{and} \quad
V_{k,{\goth p}} := \bigoplus _{i=0}^{k} \ex i{E_{{\goth p}}}\wedge V'_{k-i,{\goth p}} .$$
The goal of this note is the following theorem:

\begin{theo}\label{tint}
Let $k=1,\ldots,n$. Then $\ex k{{\goth g}}$ is the $G$-submodule of $\ex k{{\goth g}}$ 
generated by $V_{k,{\goth p}}$.  
\end{theo}
                                   
This result arises from the study of the commuting variety of ${\goth g}$
(see \cite{Ch}). One of the main step of the proof is to consider the orthogonal
complements to some subspaces of $\ex k{{\goth g}}$ in $\ex k{{\goth g}}$ with respect to
the canonical extension of the Killing form of ${\goth g}$ to $\ex k{{\goth g}}$.

\subsection{Notations}\label{int2}
$\bullet$
Let $\k^{*} := \k\setminus \{0\}$. For $E$ a finite set, its cardinality is denoted by
$\vert E \vert$. For $k,m$ positive integers, set:
$${\Bbb N}_{k}^{m} := \{(\poi j1{,\ldots,}{m}{}{}{})\in {\Bbb N}^{m} \; \vert \;
\poi j1{+\cdots +}{m}{}{}{} = k\}.$$
As usual, for $i=(\poi i1{,\ldots,}{m}{}{}{})$ in ${\Bbb N}^{m}$,
$$ \vert i \vert := \poi i1{+\cdots +}{m}{}{}{} .$$

$\bullet$ For $V$ vector space, denote by $\ex {}V$ the exterior algebra of $V$. This
algebra has a natural gradation. For $i$ integer, denote by $\ex iV$ the space of degree
$i$ of $\ex {}V$. In particular, for $i$ negative, $\ex iV$ is equal to $\{0\}$. As
${\goth g}$ is a $G$-module for the adjoint action, so is $\ex i{{\goth g}}$ for all
$i$.

\begin{lemma}\label{lint}
Let $A$ be a subgroup of $G$, $k$ a positive integer, $i$ a positive integer smaller 
than $k$, $V$ a subspace of $\ex i{{\goth g}}$ and $W$ the $A$-submodule of 
$\ex i{{\goth g}}$ generated by $V$. Then, for all $A$-submodule $W'$ of 
$\ex {k-i}{{\goth g}}$, $W\wedge W'$ is the $A$-submodule of $\ex k{{\goth g}}$ generated
by $V\wedge W'$.
\end{lemma}

\begin{proof}
Let $W''$ be the $A$-submodule of $\ex k{{\goth g}}$ generated by 
$V\wedge W'$. Let $\omega $ and $\omega '$ be in $W$ and $W'$ respectively. For some 
$\poi {\omega }1{,\ldots,}{m}{}{}{}$ in $V$ and $\poi g1{,\ldots,}{m}{}{}{}$ in $A$, 
$$ \omega = g_{1}.\omega _{1} + \cdots + g_{m}.\omega _{m},$$
whence 
$$ \omega \wedge \omega ' = g_{1}.(\omega _{1}\wedge g_{1}^{-1}.\omega ') + \cdots +
g_{m}.(\omega _{m}\wedge g_{m}^{-1}.\omega ')$$
and $W''= W\wedge W'$.
\end{proof}

$\bullet$
The Killing form of ${\goth g}$ is denoted by $\dv ..$. For $k$ positive integer, 
the Killing form of ${\goth g}$ has a natural extension to $\ex k{{\goth g}}$ and 
this extension is not degenerate.

$\bullet$ For ${\goth a}$ a semisimple Lie algebra, denote by $\b a{}$ the dimension of
its Borel subalgebras and $\j a{}$ its rank.

$\bullet$
Let ${\cal R}$ be the root system of ${\goth h}$ in ${\goth g}$, ${\cal R}_{+}$ the 
positive root system of ${\cal R}$ defined by ${\goth b}$ and $\Pi $ the basis of 
${\cal R}_{+}$. For $\alpha $ in ${\cal R}$, $H_{\alpha }$ is the coroot of $\alpha $,
the corresponding root subspace is denoted by ${\goth g}^{\alpha }$ and a generator
$x_{\alpha }$ of ${\goth g}^{\alpha }$ is chosen so that $\dv {x_{\alpha }}{x_{-\alpha }} = 1$. 

$\bullet$ We consider on $\Pi $ its structure of Dynkin diagram. As ${\goth g}$ is 
simple, $\Pi $ is connected and has three extremities when $\Pi $ has type
${\mathrm {D}}_{\rg}$, ${\mathrm {E}}_{6}$, ${\mathrm {E}}_{7}$, ${\mathrm {E}}_{8}$,
one extremity when $\Pi $ has type ${\mathrm {A}}_{1}$ and $2$ otherwise. The 
elements $\poi {\beta }1{,\ldots,}{\rg}{}{}{}$ of $\Pi $ are ordered as
in~\cite[Ch. VI]{Bou}.

$\bullet$
Let $X$ be a subset of $\Pi $. We denote by ${\cal R}_{X}$ the root subsystem of 
${\cal R}$ generated by $X$ and we set 
$$<X> := {\cal R}_{+}\cap {\cal R}_{X} \quad  \text{so that } \quad
{\cal R}_{X} = <X> \cup -<X> .$$
Let ${\goth p}_{X}$ be the parabolic subalgebra of ${\goth g}$,
$$ {\goth p}_{X} := {\goth b} \oplus 
\bigoplus _{\alpha \in <X>} {\goth g}^{-\alpha },$$
${\goth p}_{X,\u}$ its nilpotent radical, ${\goth l}_{X}$ the reductive factor of 
${\goth p}_{X}$ containing ${\goth h}$, ${\goth z}_{X}$ the center of ${\goth l}_{X}$, 
${\goth d}_{X}$ the derived algebra of ${\goth l}_{X}$, ${\goth p}_{X,-,\u}$ the 
complement to ${\goth p}_{X}$ in ${\goth g}$, invariant under $\ad {\goth h}$ and 
$E_{X}$ the sum of ${\goth z}_{X},{\goth p}_{X,\u},{\goth p}_{X,-,\u}$. When $X$ is empty,
${\goth p}_{X}$ is the Borel subalgebra ${\goth b}$.

$\bullet$ Let $X$ be a nonempty subset of $\Pi $ and $\poi X1{,\ldots,}{\n_{X}}{}{}{}$ 
its connected components. For $i=1,\ldots,\n_{X}$, denote by $n_{i}$ the cardinality of 
$<X_{i}>$ and ${\goth d}_{i}$ the subalgebra of ${\goth g}$ generated by 
${\goth g}^{\pm \beta }, \, \beta \in X_{i}$. Then 
$\poi {{\goth d}}1{,\ldots,}{\n_{X}}{}{}{}$ are the simple factors of ${\goth d}_{X}$. 
For $k$ positive integer, set:
$$ V_{k,{\goth p}_{X}} := \bigoplus _{j_{1}=0}^{n_{1}}\cdots 
\bigoplus _{j_{\n_{X}}=0}^{n_{\n_{X}}} \ex {j_{1}}{{\goth d}_{1}}\wedge \cdots \wedge 
\ex {j_{\n_{X}}}{{\goth d}_{\n_{X}}}\wedge 
\ex {k-\poi j1{-\cdots -}{\n_{X}}{}{}{}}{E_{X}}$$
and denote by $V_{k,X}$ the $G$-submodule of $\ex k{{\goth g}}$ generated by 
$V_{k,{\goth p}_{X}}$.

\section{Orthogonal complement} \label{oc}
Let $\rg \geq 2$ and $X$ a nonempty subset of $\Pi $. Set:
$$ {\goth p} := {\goth p}_{X}, \quad {\goth p}_{\u} := {\goth p}_{X,\u}, \quad 
{\goth l} := {\goth l}_{X}, \quad {\goth p}_{-,\u} := {\goth p}_{X,-,\u},$$
$$ {\goth p}_{\pm,\u} := {\goth p}_{\u} \oplus {\goth p}_{-,\u}, \quad
{\goth p}_{-} := {\goth l} \oplus {\goth p}_{-,\u}, \quad d := \dim {\goth p}_{\u} .$$

\subsection{General fact} \label{oc1}
Let $A$ be a subgroup of $G$. For $k$ positive integer and $W$ subspace of 
$\ex k{{\goth g}}$, denote by $W^{\perp}$ the orthogonal complement to $W$ in 
$\ex k{{\goth g}}$. As the bilinear form on $\ex k{{\goth g}}$, defined by the Killing 
form, is not degenerate,
$$ \dim W + \dim W^{\perp} = \dim \ex k{{\goth g}} .$$

\begin{lemma}\label{loc1}
Let $k$ be a positive integer smaller than $\dim {\goth g}$. Let $V$ be a subspace of 
$\ex k{{\goth g}}$. Denote by $W$ the $A$-submodule of $\ex k{{\goth g}}$ generated by 
$V$. Then $W^{\perp}$ is the biggest $A$-submodule of $\ex k{{\goth g}}$, contained in 
$V^{\perp}$.
\end{lemma}

\begin{proof}
Denote by $W^{\#}$ the biggest $A$-submodule contained in 
$V^{\perp}$. As $W$ is a $A$-module, so is $W^{\perp}$. Then $W^{\perp}$ is contained in 
$W^{\#}$. Moreover, $V$ is contained in the orthogonal complement to $W^{\#}$ in 
$\ex k{{\goth g}}$. Hence $W$ is orthogonal to $W^{\#}$ since the orthogonal 
complement to $W^{\#}$ is a $A$-module. As a result, $W^{\#}=W^{\perp}$.
\end{proof}

\subsection{Orthogonality} \label{oc2}
Let $V$ be a finite dimensional vector space with a non degenerate symmetric bilinear 
form on $V$. For $k$ positive integer, it induces a non degenerate symmetric bilinear
form on $\ex k{V}$. Let $\poi V1{,\ldots,}{m}{}{}{}$ be pairwise orhogonal
subspaces of $V$ such that $V$ is the direct sum of these subspaces. For
$i=(\poi i1{,\ldots,}{m}{}{}{})$ in ${\Bbb N}^{m}_{k}$, set:
$$ C_{i,V} := \ex {i_{1}}{V_{1}}\wedge \cdots \wedge \ex {i_{m}}{V_{m}} .$$
If $V_{m}$ is the direct sum of two isotropic subspaces $V_{m,+}$ and $V_{m,-}$, 
for $i=(\poi i1{,\ldots,}{m+1}{}{}{})$ in ${\Bbb N}^{m+1}_{k}$, set:
$$ i^{*} := (\poi i1{,\ldots,}{m-1}{}{}{},i_{m+1},i_{m}) \quad  \text{and} $$
$$ C'_{i,V} := \ex {i_{1}}{V_{1}}\wedge \cdots \wedge \ex {i_{m-1}}{V_{m-1}}\wedge 
\ex {i_{m}}{V_{m,+}}\wedge \ex {i_{m+1}}{V_{m,-}}.$$

\begin{lemma}\label{loc2}
Let $k$ be a positive integer.

{\rm (i)} For $i,i'$ in ${\Bbb N}^{m}_{k}$, if $i\neq i'$ then $C_{i,V}$ is orthogonal
to $C_{i',V}$.

{\rm (ii)} Suppose that $V_{m}$ is the direct sum of two isotropic subspaces $V_{m,+}$
and $V_{m,-}$. For $i,i'$ in ${\Bbb N}^{m+1}_{k}$, if $i'\neq i^{*}$ then $C'_{i,V}$ is
orthogonal to $C'_{i',V}$.
\end{lemma}

\begin{proof}
Denote by $\dv ..$ the symmetric bilinear form on $V$ and $\ex kV$. As 
$\poi V1{,\ldots,}{m}{}{}{}$ are pairwise orthogonal and $V$ is the direct sum of these 
subspaces, for $i=1,\ldots,m$, the restriction to $V_{i}\times V_{i}$ of $\dv ..$ is non 
degenerate. For $j=1,\ldots,m$, let $n_{j}$ be the dimension of the sum
$$ \poi V1{\oplus \cdots \oplus}{j}{}{}{}$$ and 
$\poi v1{,\ldots,}{n_{m}}{}{}{}$ an orthonormal basis of $V$ such that 
$\{\poi v1{,\ldots,}{n_{j}}{}{}{}\}$ is contained in the union of
$\poi V1{,\ldots,}{j}{}{}{}$ for $j=1,\ldots,m$.

(i) Let $i$ and $i'$ be in ${\Bbb N}^{m}_{k}$ such that $i\neq i'$. If $k>n_{m}$, there 
is nothing to prove. Suppose $k\leq n_{m}$. For 
$j = \poi j1{,\ldots,}{k}{}{}{}$ in $\{1,\ldots,n_{m}\}$ such that 
$1\leq \poi j1{< \cdots <}{k}{}{}{}\leq n_{m}$, set:
$$ w_{j} := \poi v{j_{1}}{\wedge \cdots \wedge }{j_{k}}{}{}{} .$$
Setting $n_{0} := 0$, $w_{j}$ is in $C_{i,V}$ if and only if
$$  \vert \{l \in \{1,\ldots,k\} \; \vert \; n_{s-1}+1\leq j_{l} \leq n_{s}\} \vert = 
i_{s} $$
for $s=1,\ldots,m$. Denote by $I_{i}$ the set of $j$ satisfying this condition so that 
$w_{j}, \, j\in I_{i}$ is a basis of $C_{i,V}$.

Let $(j,j')$ be in $I_{i}\times I_{i'}$. By definition, 
$$ \dv {w_{j}}{w_{j'}} = \det ( \dv {v_{j_{l}}}{v_{j'_{l'}}}, \; 1\leq l,l'\leq k) .$$
As the basis $\poi v1{,\ldots,}{n_{m'}}{}{}{}$ is orthonormal, 
$$ \dv {v_{j_{l}}}{v_{j'_{l'}}} = \delta _{j_{l},j'_{l'}}$$
with $\delta _{s,s'}$ the Kronecker symbol. As a result, if all the lines of the above 
matix are all different from $0$ then
$$ \vert \{l \in \{1,\ldots,k\} \; \vert \; n_{s-1}+1\leq j_{l} \leq n_{s}\} \vert = 
\vert \{l \in \{1,\ldots,k\} \; \vert \; n_{s-1}+1\leq j'_{l} \leq n_{s}\} \vert $$
for $s=1,\ldots,m$ since $\poi V1{,\ldots,}{m}{}{}{}$ are pairwise orthogonal. Then
$\dv {w_{j}}{w_{j'}} = 0$ since $i\neq i'$, whence the assertion.

(ii) Let $i$ and $i'$ be in ${\Bbb N}^{m+1}_{k}$ such that $i'\neq i^{*}$. By (i), 
we can suppose that $i_{s}=i'_{s}$ for $s=1,\ldots,m-1$. Since $V_{m,+}$ and 
$V_{m,-}$ are isotropic, they have the same dimension $m_{0}$ and $V_{m}$ has a 
basis $\poi u1{,\ldots,}{2m_{0}}{}{}{}$ such that 
$$ \{\poi u1{,\ldots,}{m_{0}}{}{}{}\} \subset V_{m,+}, \quad
\{\poi u{m_{0}+1}{,\ldots,}{2m_{0}}{}{}{}\} \subset V_{m,-}, \quad 
\dv {u_{s}}{u_{s'+m_{0}}} = \delta _{s,s'}$$
for $1\leq s,s'\leq m_{0}$. Let $\poie v1{,\ldots,}{{n_{m}}}{}{}{}{\prime}{\prime}$ be 
the basis of $V$ such that $v'_{l}=v_{l}$ for $l=1,\ldots,n_{m-1}$ and  
$v'_{l'}=u_{l'-n_{m-1}}$ for $l'=n_{m-1}+1,\ldots,n_{m}$. For 
$j=\poi j1{,\ldots,}{k}{}{}{}$ in $\{1,\ldots,n_{m}\}$ such that 
$1\leq \poi j1{<\cdots <}{k}{}{}{}\leq n_{m}$, set:
$$ w'_{j} = v'_{j_{1}}\wedge \cdots \wedge v'_{j_{k}} .$$
Then $w'_{j}$ is in $C'_{i,V}$ if and only if
$$ \vert \{l \in \{1,\ldots,k\} \; \vert \; n_{m-1}+1\leq j_{l} \leq n_{m-1}+m_{0}\}\vert
= i_{m} , \quad 
\vert \{l \in \{1,\ldots,k\} \; \vert \; n_{m-1}+m_{0}+1\leq j_{l} \leq n_{m}\} \vert = 
i_{m+1} , \quad$$
$$ \vert \{l \in \{1,\ldots,k\} \; \vert \; n_{s-1}+1\leq j_{l} \leq n_{s}\} \vert = 
i_{s} $$
for $s=1,\ldots,m-1$. Denote by $I_{i}$ the set of $j$ satisfying this condition so that 
$w'_{j}, \, j\in I_{i}$ is a basis of $C'_{i,V}$.

Let $(j,j')$ be in $I_{i}\times I_{i'}$. By definition, 
$$ \dv {w'_{j}}{w'_{j'}} = \det ( \dv {v'_{j_{l}}}{v'_{j'_{l'}}}, \; 1\leq l,l'\leq k) .$$
Then 
$$ j_{l} \leq n_{m-1} \Longrightarrow 
\dv {v'_{j_{l}}}{v'_{j'_{l'}}} = \delta _{j_{j},j'_{l'}},$$
$$ j_{l} > n_{m-1} \quad  \text{and} \quad j'_{l'} > n_{m-1} \Longrightarrow 
\dv {v'_{j_{l}}}{v'_{j'_{l'}}} = \delta _{\vert j_{j}-j'_{l'}\vert,m_{0}} .$$
As a result, if all the lines of the above matix are all different from $0$ then
$$ \vert \{l \in \{1,\ldots,k\} \; \vert \; n_{m-1}+1\leq j_{l} \leq n_{m-1}+m_{0}\} 
\vert = 
\vert \{l \in \{1,\ldots,k\} \; \vert \; n_{m-1}+m_{0}+1\leq j'_{l} \leq n_{m}\} \vert $$
since $i_{s}=i'_{s}$ for $s\leq m-1$ and $\poi V1{,\ldots,}{m}{}{}{}$ are pairwise 
orthogonal. Then $\dv {w_{j}}{w_{j'}} = 0$ since $i'\neq i^{*}$, whence the 
assertion.
\end{proof}

For $i = (i_{1},i_{2},i_{3})$ in ${\Bbb N}^{3}$, set:
$$ C_{i} := \ex {i_{1}}{{\goth l}}\wedge \ex {i_{2}}{{\goth p}_{\u}}\wedge 
\ex {i_{3}}{{\goth p}_{-,\u}} $$
and denote by $i^{*}$ the element $(i_{1},i_{3},i_{2})$ of ${\Bbb N}^{3}$.

\begin{coro}\label{coc2}
Let $k$ be a positive integer.

{\rm (i)} For $i,i'$ in ${\Bbb N}^{3}_{k}$, $C_{i}$ is orthogonal to $C_{i'}$ if 
$i^{*}\neq i'$.

{\rm (ii)} For $i$ in ${\Bbb N}^{3}_{k}$, the orthogonal complement to $C_{i}$ in 
$\ex k{{\goth g}}$ is equal to
$$ \bigoplus _{i'\in {\Bbb N}^{3}_{k}\setminus \{i^{*}\}} C_{i'} .$$
\end{coro}

\begin{proof}
(i) Let $i$ and $i'$ be in ${\Bbb N}^{3}_{k}$ such that $i^{*}\neq i'$. By 
Lemma~\ref{loc1}(ii), $C_{i}$ is orthogonal to $C_{i'}$ since ${\goth l}$ and 
${\goth p}_{\pm,\u}$ are orthogonal and ${\goth p}_{\u}$ and ${\goth p}_{-,\u}$ are 
isotropic. 

(ii) Since $\ex k{{\goth g}}$ is the direct sum of 
$C_{i'}, \, i'\in {\Bbb N}^{3}_{k}$, 
the orthogonal complement to $C_{i}$ in $\ex k{{\goth g}}$ is the direct sum of 
$C_{i'}, \, i' \in {\Bbb N}^{\n+3}_{k}\setminus \{i^{*}\}$ by (i).
\end{proof}

\begin{coro}\label{c2oc2}
Let $k=1,\ldots,d$.

{\rm (i)} The orthogonal complement to $\ex k{{\goth p}_{-,\u}}$ in $\ex k{{\goth g}}$ is 
equal to ${\goth p}_{-}\wedge \ex {k-1}{{\goth g}}$.

{\rm (ii)} The orthogonal complement to $\ex k{{\goth p}_{\pm,\u}}$ in $\ex k{{\goth g}}$
is equal to ${\goth l}\wedge \ex {k-1}{{\goth g}}$.
\end{coro}

\begin{proof}
(i)  Let $I_{1}$ be the subset of ${\Bbb N}^{3}_{k}$,
$$ I_{1} :=\{(i_{1},i_{2},i_{3}) \in {\Bbb N}^{3}_{k} \; \vert \; i_{1}=i_{2}=0\} .$$
The complement to $I_{1}^{*}$ in ${\Bbb N}_{k}^{3}$ is equal to 
$$ \{(i_{1},i_{2},i_{3}) \in {\Bbb N}^{3}_{k} \; \vert \; i_{1}>0 \quad  \text{or} \quad
i_{3}>0\} ,$$
whence the assertion by Corollary~\ref{coc2} since 
$$ \ex k{{\goth p}_{-,\u}} = \bigoplus _{i\in I_{1}} C_{i} .$$

(ii) Let $I_{2}$ be the subset of ${\Bbb N}^{3}_{k}$,
$$ I_{2} :=\{(i_{1},i_{2},i_{3}) \in {\Bbb N}^{3}_{k} \; \vert \; i_{1}=0\} .$$
The complement to $I_{2}^{*}$ in ${\Bbb N}_{k}^{3}$ is equal to 
$$ \{(i_{1},i_{2},i_{3}) \in {\Bbb N}^{3}_{k} \; \vert \; i_{1}>0\} ,$$
whence the assertion by Corollary~\ref{coc2} since $\ex k{{\goth p}_{\pm,\u}}$ is the 
sum of $C_{i}, \, i \in I_{2}$.
\end{proof}

\section{Action of the unipotent radical of a parabolic subgroup}\label{au}
Let $\rg \geq 2$ and $X$ a subset of $\Pi $. Set: 
$$ {\goth p} := {\goth p}_{X}, \quad {\goth p}_{\u} := {\goth p}_{X,\u}, \quad 
{\goth l} := {\goth l}_{X}, \quad {\goth p}_{-,\u} := {\goth p}_{X,-,\u}, $$
$$ {\goth p}_{\pm,\u} := {\goth p}_{\u} \oplus {\goth p}_{-,\u} , \quad
{\goth p}_{-} := {\goth l}\oplus {\goth p}_{-,\u}, \quad d := \dim {\goth p}_{\u} .$$
Denote by $L$ and $H$ the connected closed subgroups of $G$ whose Lie algebras are
${\goth l}$ and ${\goth h}$ respectively. Let $P$ and $P_{-}$ be the normalizers of
${\goth p}$ and ${\goth p}_{-}$ in $G$ and $P_{\u}$ and $P_{-,\u}$ their unipotent
radicals. 

\subsection{Invariant subspaces} \label{au1}
Let $k=1,\ldots,d$, $W_{k}$ the biggest $P_{\u}$-submodule of $\ex k{{\goth g}}$ contained
in ${\goth p}_{-}\wedge \ex {k-1}{{\goth g}}$ and $V_{k,\u}$ the $P_{\u}$-submodule of
$\ex k{{\goth g}}$ generated by $\ex {k}{{\goth p}_{-,\u}}$.

\begin{lemma}\label{lau1}
  Let $W_{k,0}$ be the subspace of elements of
  ${\goth p}_{-}\wedge \ex {k-1}{{\goth g}}$ invariant under ${\goth u}$.

{\rm (i)} The subspace $W_{k}$ of $\ex k{{\goth g}}$ is invariant under ${\goth u}$.

{\rm (ii)} The subspace $W_{k,0}$ of $\ex k{{\goth g}}$ is contained in $W_{k}$ and 
generated by highest weight vectors.

{\rm (iii)} The subspace $\ec Uu{}-{}.W_{k,0}$ of $\ex k{{\goth g}}$ is the 
biggest $G$-submodule of $\ex k{{\goth g}}$ contained in 
${\goth p}_{-}\wedge \ex {k-1}{{\goth g}}$.
\end{lemma}

\begin{proof}
(i) Denote by $W'_{k}$ the $L$-submodule of $\ex k{{\goth g}}$ generated by $W_{k}$. As
${\goth p}_{-}\wedge \ex {k-1}{{\goth g}}$ is invariant under $L$, $W'_{k}$ is 
contained in ${\goth p}_{-}\wedge \ex {k-1}{{\goth g}}$. For $x$ in ${\goth p}_{\u}$ and
$g$ in $L$, 
$$ x.g.W_{k} = g.\Ad g^{-1}(x).W_{k} \subset W'_{k},$$
since ${\goth p}_{\u}$ is invaraint under the adjoint action of $L$ in ${\goth g}$. Then 
$W'_{k}$ is invariant under $P_{\u}$, whence $W_{k}=W'_{k}$. As a result, $W_{k}$ is a
${\goth u}$-submodule of $\ex k{{\goth g}}$ since ${\goth u}$ is contained in ${\goth p}$.

(ii) For $\omega $ in $W_{k,0}$, the subspace of $\ex k{{\goth g}}$ generated by 
$\omega $ is a $P_{\u}$-submodule contained in 
${\goth p}_{-}\wedge \ex {k-1}{{\goth g}}$.  Hence $W_{k,0}$ is contained in $W_{k}$. 
Moreover, for $x$ in ${\goth u}$ and $g$ in $H$, 
$$ x.g.\omega = g.\Ad g^{-1}(x).\omega  = 0 .$$ 
Hence $W_{k,0}$ is invariant under $H$. As a result, $W_{k,0}$ is generated by 
highest weight vectors.

(iii) By (ii), $\ec Uu{}-{}.W_{k,0}$ is the $G$-submodule of $\ex k{{\goth g}}$ generated
by $W_{k,0}$. As ${\goth u}_{-}$ is contained in ${\goth p}_{-}$, 
${\goth p}_{-}\wedge \ex {k-1}{{\goth g}}$ is a $\ec Uu{}-{}$-submodule of 
$\ex k{{\goth g}}$ so that $\ec Uu{}-{}.W_{k,0}$ is contained in 
${\goth p}_{-}\wedge \ex {k-1}{{\goth g}}$. Since a $G$-submodule of $\ex k{{\goth g}}$
is generated by highest weight vectors, $\ec Uu{}-{}.W_{k,0}$ is the biggest 
$G$-submodule of $\ex k{{\goth g}}$ contained in 
${\goth p}_{-}\wedge \ex {k-1}{{\goth g}}$.
\end{proof}

\begin{coro}\label{cau1}
{\rm (i)} The subspace $W_{k}$ of $\ex {k}{{\goth g}}$ is the biggest $G$-submodule
contained in ${\goth p}_{-}\wedge \ex {k-1}{{\goth g}}$. 

{\rm (ii)} The subspace $V_{k,\u}$ of $\ex k{{\goth g}}$ is a $G$-submodule of 
$\ex k{{\goth g}}$.
\end{coro}

\begin{proof}
(i) Denote by $\widetilde{W}_{k}$ the biggest $G$-submodule of $\ex k{{\goth g}}$ 
contained in ${\goth p}_{-}\wedge \ex {k-1}{{\goth g}}$. Then $\widetilde{W}_{k}$ is 
contained in $W_{k}$. Let $W_{k,1}$ be a complement to $\widetilde{W}_{k}$ in
$\ex k{{\goth g}}$, invariant under $G$. Then $W_{k}$ is the direct sum of
$\widetilde{W}_{k}$ and $W_{k}\cap W_{k,1}$. By Lemma~\ref{lau1}(i),
$W_{k}\cap W_{k,1}$ is invariant under ${\goth u}$. Then, by Lie's Theorem,
$W_{k,0}\cap W_{k,1} \neq \{0\}$ if $W_{k}\cap W_{k,1}\neq \{0\}$. Hence
$W_{k}=\widetilde{W}_{k}$ since $W_{k,0}$ is contained in $\widetilde{W}_{k}$ by
Lemma~\ref{lau1}(iii).

(ii) By Corollary~\ref{c2oc2}(i) and Lemma~\ref{loc1}, $W_{k}$ is the orthogonal
complement to $V_{k,\u}$ in $\ex k{{\goth g}}$. Hence $V_{k,\u}$ is a $G$-module by (i). 
\end{proof}

\subsection{A particular case} \label{au2}
In this subsection, for some $\beta $ in $\Pi $, $X := \Pi \setminus \{\beta \}$.
Denote by ${\goth h}_{\beta }$ the orthogonal complement to $H_{\beta }$ in ${\goth h}$.
Let $Z$ be the subset of elements $\alpha $ of $<X>$ such that $\beta +\alpha $ is a
root. Set:
$$ Y := {\cal R}_{+}\setminus (<X>\cup \{\beta \}), \quad
Z' := <X>\setminus Z, \quad E := \bigoplus _{\alpha \in Y} {\goth g}^{\alpha }, \quad
E_{-} := \bigoplus _{\alpha \in Y} {\goth g}^{-\alpha } , $$
$$ {\goth u}_{0} := \bigoplus _{\alpha \in Z} {\goth g}^{\alpha }, \quad
{\goth u}_{0,+} := \bigoplus _{\alpha \in Z'} {\goth g}^{\alpha }, \quad
{\goth u}_{0,0} := \bigoplus _{\alpha \in Z} {\goth g}^{-\alpha }, \quad
{\goth u}_{0,-} := \bigoplus _{\alpha \in Z'} {\goth g}^{-\alpha } .$$
Then
$$ {\goth g} := E_{-} \oplus {\goth g}^{-\beta } \oplus {\goth u}_{0,-} \oplus
{\goth u}_{0,0} \oplus \k H_{\beta } \oplus {\goth h}_{\beta } \oplus {\goth u}_{0} \oplus
{\goth u}_{0,+}\oplus {\goth g}^{\beta } \oplus E .$$
For $i=(\poi i1{,\ldots,}{10}{}{}{})$, set:
$$ C_{i} := \ex {i_{1}}{E_{-}}\wedge \ex {i_{2}}{{\goth g}^{-\beta }}\wedge
\ex {i_{3}}{{\goth u}_{0,-}}\wedge \ex {i_{4}}{{\goth u}_{0,0}}\wedge
\ex {i_{5}}{\k H_{\beta }}\wedge $$
$$ \ex {i_{6}}{{\goth h}_{\beta }}\wedge
\ex {i_{7}}{{\goth u}_{0}}\wedge \ex {i_{8}}{{\goth u}_{0,+}}\wedge
\ex {i_{9}}{{\goth g}^{\beta }}\wedge \ex {i_{10}}{E} .$$
For $k$ positive integer, $\ex k{{\goth g}}$ is the direct sum of
$C_{i}, \, i\in {\Bbb N}^{10}_{k}$.

For $\alpha $ in $Z$, denote by $\omega '_{\alpha }$ and $\omega _{\alpha }$ the elements
of $\ex 2{{\goth g}}$,
$$ \omega '_{\alpha } := H_{\beta }\wedge [x_{\beta },x_{\alpha }] +
2 x_{\beta }\wedge x_{\alpha }, \quad
\omega _{\alpha } := H_{\beta }\wedge [x_{-\beta },x_{-\alpha }] +
c_{\alpha } x_{-\beta }\wedge x_{-\alpha }, \quad  \text{with} \quad$$
$$ c_{\alpha } := -\frac{1}{2} \dv {H_{\beta }}{H_{\beta }}
\dv {[x_{\beta },x_{\alpha }]}{[x_{-\beta },x_{-\alpha }]}$$
so that $\omega _{\alpha }$ is orthogonal to $\omega '_{\alpha }$.

\begin{lemma}\label{lau2}
Let $k=1,\ldots,d$. Denote by $I$ the subset of elements $i$ of ${\Bbb N}^{10}_{k}$
such that $\poi i1{+\cdots +}{8}{}{}{}\geq 2$, $M'$ the subspace of elements
$\mu _{\alpha }, \, \alpha \in Z$ of ${\ex {k-1}E}^{Z}$ such that
$$ \sum_{\alpha \in Z} [x_{\beta },x_{\alpha }]\wedge \mu _{\alpha } = 0$$
and $M$ the image of $M'$ by the map
$$ \xymatrix{{\ex {k-1}E}^{Z} \ar[rr] && \ex k{{\goth g}}}, \qquad
(\mu _{\alpha }, \, \alpha \in Z) \longmapsto \sum_{\alpha \in Z} x_{\alpha }\wedge
\mu _{\alpha } .$$
The space $W_{k}$ is contained in the subspace of $\ex k{{\goth g}}$ generated by
$M$, $\omega '_{\alpha }\wedge \ex {k-2}E, \, \alpha \in Z$, $C_{i}, \, i \in I$.
\end{lemma}

\begin{proof}
By Corollary~\ref{c2oc2}(i) and Corollary~\ref{cau1}, $W_{k}$ is the biggest
$G$-module contained in ${\goth p}_{-}\wedge \ex {k-1}{{\goth g}}$. Denoting by $I'$ the
subset of elements $i$ of ${\Bbb N}^{10}_{k}$ such that
$$ \poi i1{+\cdots +}{8}{}{}{} > 0 ,$$
${\goth p}_{-}\wedge \ex {k-1}{{\goth g}}$ is the sum of $C_{i}, \, i \in I'$. Then for
$i$ in $I$ and $x$ in ${\goth g}$, $x.C_{i}$ is contained in
${\goth p}_{-}\wedge \ex {k-1}{{\goth g}}$. The complement to $I$ in $I'$ is equal to
the subset of elements $i$ of $I'$ such that $i_{9}+i_{10} = k-1$. For $i$ in
$I'\setminus I$ such that $i_{5}=i_{7}=0$, $x_{\beta }.C_{i}$ is contained in
${\goth p}_{-}\wedge \ex {k-1}{{\goth g}}$ since
$$ [x_{\beta },E_{-}] \subset {\goth p}_{-}, \quad
[x_{\beta },{\goth g}^{-\beta }] \subset {\goth p}_{-}, \quad
[x_{\beta },{\goth u}_{0,-}+{\goth u}_{0,0}+{\goth h}_{\beta }+{\goth u}_{0,+}]  
= \{0\} . $$
For $i$ in $I'\setminus I$,
$$ i_{7} = 1 \quad  \text{and} \quad i_{9} = 0 \Longrightarrow
x_{\beta }.C_{i} \subset C_{i} \oplus \ex kE, $$
$$ i_{7} = 1 \quad  \text{and} \quad i_{9} = 1 \Longrightarrow
x_{\beta }.C_{i} \subset C_{i} \oplus {\goth g}^{\beta }\wedge \ex {k-1}E,$$
$$ i_{5} = 1 \quad  \text{and} \quad i_{9} = 0 \Longrightarrow
x_{\beta }.C_{i} \subset C_{i} \oplus {\goth g}^{\beta }\wedge \ex {k-1}E, $$
$$ i_{5} = 1 \quad  \text{and} \quad i_{9} = 1 \Longrightarrow
x_{\beta }.C_{i} \subset C_{i} .$$
As a result, for $\omega $ and $\mu _{\alpha }, \, \alpha \in Z$ in $\ex {k-1}E$ and
$\mu '_{\alpha }, \, \alpha \in Z$ in $\ex {k-2}E$ such that
$$ \omega ' + H_{\beta }\wedge \omega + \sum_{\alpha \in Z} x_{\alpha }
\wedge \mu _{\alpha }
+ \sum_{\alpha \in Z} x_{\beta }\wedge x_{\alpha }\wedge \mu '_{\alpha } \in W_{k}$$
for some $\omega '$ in the sum of $C_{i}, \, i\in I$,
$$ -2\omega + \sum_{\alpha \in Z} [x_{\beta },x_{\alpha }]\wedge \mu '_{\alpha } = 0
\quad  \text{and} \quad  
\sum_{\alpha \in Z} [x_{\beta },x_{\alpha }]\wedge \mu _{\alpha } = 0$$
since $x_{\beta }.W_{k}$ is contained in ${\goth p}_{-}\wedge \ex {k-1}{{\goth g}}$,  
whence the lemma.
\end{proof}

For $i=(\poi i1{,\ldots,}{10}{}{}{})$ in ${\Bbb N}^{10}$, denote by $i^{*}$ the element
of ${\Bbb N}^{10}$,
$$ i^{*} := (i_{10},i_{9},i_{8},i_{7},i_{5},i_{6},i_{4},i_{3},i_{2},i_{1}) .$$
By Lemma~\ref{loc2}(ii), for $k$ positive integer and
$i,i'$ in ${\Bbb N}^{10}_{k}$, $C_{i}$ is orthogonal to $C_{i'}$ if and only if
$i'\neq i^{*}$. 

\begin{coro}\label{cau2} 
Let $k=1,\ldots,d$ and $\alpha $ in $Z$.

{\rm (i)} The space $\omega _{\alpha }\wedge \ex {k-2}{E_{-}}$ is contained in $V_{k,\u}$.

{\rm (ii)} The space ${\goth g}^{\alpha }\wedge \ex {k-1}{E_{-}}$ is contained in
$V_{k,\u}$. 

{\rm (iii)} The spaces $H_{\alpha }\wedge \ex {k-2}{E_{-}}$ and
${\goth g}^{-\alpha }\wedge \ex {k-1}{E_{-}}$ are contained in $V_{k,\u}$.

{\rm (iv)} The space ${\goth h}_{\beta }\wedge \ex {k-1}{E_{-}}$ is contained in 
$V_{k,\u}$.
\end{coro}

\begin{proof}
(i) Let $I_{1}$ be the subset of elements $i$ of ${\Bbb N}^{10}_{k}$ such that
$$ (i_{1} = k-1 \quad  \text{and} \quad i_{5} = 1) \quad  \text{or} \quad
(i_{1}=k-2, \, i_{2} = 1, \, i_{4} = 1) .$$
Then $\omega _{\alpha }\wedge \ex {k-2}{E_{-}}$ is contained in the
sum of $C_{i}, \, i\in I_{1}$. Hence $\omega _{\alpha }\wedge \ex {k-2}{E_{-}}$ is
orthogonal to $C_{i}$ for all $i$ in $I$ and for $i$ such that $i_{7}=1$ and
$i_{10}=k-1$. By Corollary~\ref{coc2}(i), for all $\gamma $ in $Z$,
$\omega _{\alpha }\wedge \ex {k-2}{E_{-}}$ is orthogonal to
$\omega '_{\gamma }\wedge \ex {k-2}{E}$ since $\omega _{\alpha }$ and $\omega '_{\gamma }$
are orthogonal, whence the assertion by Lemma~\ref{lau2} since $V_{k,\u}$ is the
orthogonal complement to $W_{k}$ in $\ex k{{\goth g}}$ by Lemma~\ref{loc1}.

(ii) The space ${\goth g}^{\alpha }\wedge \ex {k-1}{E_{-}}$ is contained
in $C_{i}$ with $i$ in ${\Bbb N}^{10}_{k}$ such that $i_{7}=1$ and $i_{1}=k-1$. Hence
${\goth g}^{\alpha }\wedge \ex {k-1}{E_{-}}$ is orthogonal to $C_{j}$ for all $j$
in $I$. Moreover, it is orthogonal to $C_{j}$ for $j$ in ${\Bbb N}^{10}_{k}$ such that
$j_{7}=1$ and $j_{10}=k-1$ and $C_{j'}$ for $j'$ in $I_{1}$. As a result,
${\goth g}^{\alpha }\wedge \ex {k-1}{E_{-}}$ is orthogonal to $W_{k}$ by Lemma~\ref{lau2},
whence the assertion since $V_{k,\u}$ is the orthogonal complement to $W_{k}$ in
$\ex k{{\goth g}}$ by Lemma~\ref{loc1}.

(iii) By (ii) and Corollary~\ref{cau1}(ii), for $\omega $ in $\ex {k-1}{E_{-}}$,
$$ V_{k,\u} \ni x_{-\alpha }.(x_{\alpha }\wedge \omega ) = -H_{\alpha }\wedge \omega +
x_{\alpha }\wedge x_{-\alpha }.\omega \quad  \text{and} $$
$$  x_{-\alpha }.(H_{\alpha }\wedge \omega ) = 2x_{-\alpha }\wedge \omega +
H_{\alpha }\wedge x_{-\alpha }.\omega .$$
As $E_{-}$ is invariant under the adjoint action of $x_{-\alpha }$,
$x_{\alpha }\wedge x_{-\alpha }.\omega $ is in $V_{k,\u}$ by (ii), whence the assertion.

(iv) The space ${\goth h}_{\beta }\wedge \ex {k-1}{E_{-}}$ is equal to $C_{i}$ for
$i$ such that $i_{1}=k-1$ and $i_{6}=1$. Then $C_{i}$ is orthogonal to $C_{j}$ for
$j$ in $I$. Moreover, it is orthogonal to ${\goth g}^{\alpha }\wedge \ex {k-1}E$ for all
$\alpha $ in $Z$ and $H_{\beta }\wedge \ex {k-1}{{\goth p}_{\u}}$ since 
${\goth h}_{\beta }$ is orthogonal to $H_{\beta }$ and ${\goth u}$. As a result,
${\goth h}_{\beta }\wedge \ex {k-1}{E_{-}}$ is orthogonal to $W_{k}$ by Lemma~\ref{lau2},
whence the assertion since $V_{k,\u}$ is the orthogonal complement to $W_{k}$ in
$\ex k{{\goth g}}$ by Lemma~\ref{loc1}.
\end{proof}

Denote by ${\goth d}$ the derived algebra of ${\goth l}$.

\begin{prop}\label{pau2}
Let $k=1,\ldots,d$, $i=0,\ldots,k-1$.

{\rm (i)} Let $M$ be a $P_{\u}$-submodule of $\ex {i}{{\goth g}}$. Then the
$P_{\u}$-submodule of $\ex k{{\goth g}}$ generated by $
\ex {k-i}{{\goth p}_{-,\u}}\wedge M$ contains
$\ex {k-i-1}{{\goth p}_{-,\u}}\wedge {\goth d}\wedge M$.

{\rm (ii)} Let $N$ be a $P_{-,\u}$-submodule of $\ex {i}{{\goth g}}$. Then the
$P_{-,\u}$-submodule of $\ex k{{\goth g}}$ generated by
$\ex {k-i}{{\goth p}_{\u}}\wedge N$ contains
 $\ex {k-i-1}{{\goth p}_{\u}}\wedge {\goth d}\wedge N$.
\end{prop}

\begin{proof}
(i) By Lemma~\ref{lint}, it is sufficient to prove that $V_{k-i,\u}$ contains
$\ex {k-i-1}{{\goth p}_{-,\u}}\wedge {\goth d}$ since $M$ is a $P_{\u}$-module.
For $\alpha $ in $Z$,
$$ H_{\alpha } \in \frac{\beta (H_{\alpha })}{2} H_{\beta } + {\goth h}_{\beta } .$$
So, by Corollary~\ref{cau2},(iii) and (iv), $H_{\beta }\wedge \ex {k-i-1}{E_{-}}$ is
contained in $V_{k-i,\u}$. Then, by Corollary~\ref{cau2}(i),
$$ V_{k-i,\u} \supset {\goth g}^{-\beta }\wedge {\goth g}^{-\alpha }\wedge
\ex {k-i-2}{E_{-}} .$$
As a result, by Corollary~\ref{cau2}(iii), for all $\alpha $ in $Z$,
${\goth g}^{-\alpha }\wedge \ex {k-i-1}{{\goth p}_{-,\u}}$ is contained in
$V_{k-i,\u}$. As ${\goth g}$ is simple, the ${\goth l}$-submodule of ${\goth g}$,
generated by ${\goth g}^{\alpha }, \, \alpha \in Z$ is equal to ${\goth d}$. By
Corollary~\ref{cau1}(ii), $V_{k-i,\u}$ is a ${\goth l}$-module. Then, by
Lemma~\ref{lint}, $V_{k-i,\u}$ contains $\ex {k-i-1}{{\goth p}_{-,\u}}\wedge {\goth d}$
since $\ex {k-i-1}{{\goth p}_{-,\u}}$ is a ${\goth l}$-module, whence the assertion.

(ii) For some automorphism $g$ of ${\goth g}$, $g({\goth p}_{\u})={\goth p}_{-,\u}$,
$g({\goth p}_{-,\u})={\goth p}_{\u}$, $g({\goth h})={\goth h}$. Then ${\goth l}$ and
${\goth d}$ are invariant under $g$. A a result, by (i), the $P_{-,\u}$-submodule of
$\ex k{{\goth g}}$ generated by  $\ex {k-i}{{\goth p}_{\u}}\wedge N$ contains
$\ex {k-i-1}{{\goth p}_{\u}}\wedge {\goth d}\wedge N$ since $g(N)$ is a
$P_{\u}$-submodule of $\ex i{{\goth g}}$.
\end{proof}

\section{Proof of Theorem~\ref{tint}} \label{pt}
Let $\rg \geq 2$ and $X$ a nonempty subset of $\Pi $, different from $\Pi $. Set: 
$$ {\goth p} := {\goth p}_{X}, \quad {\goth p}_{\u} := {\goth p}_{X,\u}, \quad 
{\goth l} := {\goth l}_{X}, \quad {\goth z} := {\goth z}_{X}, \quad
{\goth d} := {\goth d}_{X}, \quad \n := \n_{X}, $$
$$ {\goth p}_{-,\u} := {\goth p}_{X,-,\u}, \quad
{\goth p}_{\pm,\u} := {\goth p}_{\u} \oplus {\goth p}_{-,\u} , \quad
E := E_{X}, \quad {\goth p}_{-} := {\goth l}\oplus {\goth p}_{-,\u}, \quad
d := \dim {\goth p}_{\u}.$$
Recall that $\poi {{\goth d}}1{,\ldots,}{\n}{}{}{}$ are the simple factors of
${\goth d}$ and for $i=1,\ldots,\n$, $n_{i}$ is the number of positive roots $\alpha $
such that ${\goth g}^{\alpha }$ is contained in ${\goth d}_{i}$. Let $P_{\u}$ and
$P_{-,\u}$ be as in Section~\ref{au}. For $k=1,\ldots,n$, set
$V_{k,{\goth p}} := V_{k,{\goth p}_{X}}$ and $V_{k,X} := V_{k}$. 

\subsection{A partial result} \label{pt1}
Let $n'$ be the sum $\poi n1{+\cdots +}{\n}{}{}{}$. For $k=1,\ldots,n$, denote by
$V'_{k}$ the subspace of $\ex k{{\goth g}}$,
$$ V'_{k} := \bigoplus _{j=0}^{n'} \ex j{{\goth d}}\wedge \ex {k-j}E .$$

\begin{prop}\label{ppt1}
Suppose $\vert X \vert = \rg-1$. Let $k=1,\ldots,n$. Then $\ex k{{\goth g}}$ is
the $G$-submodule of $\ex k{{\goth g}}$ generated by $V'_{k}$.
\end{prop}

\begin{proof}
For $k=1,\ldots,d$, denote by $E_{k}$ the $G$-submodule of $\ex k{{\goth g}}$ generated
by $\ex kE$. For $k=1,\ldots,n'$, $V'_{k} = \ex k{{\goth g}}$ and for $k>n'$,
$$ V'_{k} = \ex {n'}{{\goth g}}\wedge \ex {k-n'}E .$$
So, by Lemma~\ref{lint}, it is sufficient to prove that $E_{k}=\ex k{{\goth g}}$
for $k=1,\ldots,d$ since $n=n'+d$. 
  
Prove the assertion by induction on $k$. For $k=1$ the assertion is true since
${\goth g}$ is simple. Suppose $k>1$ and the assertion true for $k-1$. As
$\vert X \vert = \rg -1$, ${\goth z}$ has dimension $1$ and
$$ \ex kE = \ex k{{\goth p}_{\pm,\u}} \oplus {\goth z}\wedge
\ex {k-1}{{\goth p}_{\pm,\u}} = $$
$$ \bigoplus _{j=0}^{k} \ex {k-j}{{\goth p}_{-,\u}}\wedge \ex {j}{{\goth p}_{\u}}
\oplus \bigoplus _{j=0}^{k-1}
\ex {k-1-j}{{\goth p}_{-,\u}}\wedge {\goth z}\wedge \ex {j}{{\goth p}_{\u}} .$$
For $j=0,\ldots,d$, $\ex {j}{{\goth p}_{\u}}$ and
$\ex {j+1}{{\goth p}_{\u}} + {\goth z}\wedge \ex {j}{{\goth p}_{\u}}$ are
$P_{\u}$-submodules of $\ex j{{\goth g}}$ and $\ex {j+1}{{\goth g}}$ respectively.
Then, by Proposition~\ref{pau2}(i), $E_{k}$ contains
$$ \ex {k-j-1}{{\goth p}_{-,\u}}\wedge {\goth d}\wedge \ex {j}{{\goth p}_{\u}}
\quad  \text{and} \quad
\ex {k-j'-2}{{\goth p}_{-,\u}}\wedge {\goth z}\wedge {\goth d}\wedge
\ex {j'}{{\goth p}_{\u}}$$
for $j=0,\ldots,k-1$ and $j'=0,\ldots,k-2$. Hence $E_{k}$ contains
$$ \ex {k-j-1}{{\goth p}_{-,\u}}\wedge {\goth g}\wedge \ex {j}{{\goth p}_{\u}}
\quad  \text{and} \quad
\ex {k-j'-2}{{\goth p}_{-,\u}}\wedge {\goth z}\wedge {\goth g} 
\wedge \ex {j'}{{\goth p}_{\u}}$$
for $j=0,\ldots,k-1$ and $j'=0,\ldots,k-2$ since $\ex kE$ contains
$$ \ex {k-j}{{\goth p}_{-,\u}}\wedge \ex {j}{{\goth p}_{\u}}, \quad
\ex {k-j-1}{{\goth p}_{-,\u}}\wedge \ex {j+1}{{\goth p}_{\u}}, $$
$$ \ex {k-j-1}{{\goth p}_{-,\u}}\wedge {\goth z}\wedge \ex {j}{{\goth p}_{\u}}, \quad
\ex {k-j'-2}{{\goth p}_{-,\u}}\wedge {\goth z}\wedge \ex {j'+1}{{\goth p}_{\u}}$$
for $j=0,\ldots,k-1$ and $j'=0,\ldots,k-2$. As a result, $E_{k}$ contains
${\goth g}\wedge \ex {k-1}E$. Then, by Lemma~\ref{lint} and the induction hypothesis,
$E_{k}=\ex k{{\goth g}}$, whence the proposition.
\end{proof}

\begin{rema}\label{rpt1}
When $X$ is connected, $V_{k,{\goth p}}=V'_{k,{\goth p}}$. Then, in this case under the
assumption $\vert X \vert = \rg -1$, $V_{k}=\ex k{{\goth g}}$ by
Proposition~\ref{ppt1}. 
\end{rema}

\subsection{A first particular case} \label{pt2}
In this subsection, $\vert X \vert = \rg -1$ so that $\n \in \{1,2,3\}$. As a matter of
fact, $\n=3$ only for type ${\mathrm {D}}$ and ${\mathrm {E}}$. As in
Subsection~\ref{pt1}, $n'=\poi n1{+\cdots +}{\n}{}{}{}$. For
$i=(\poi i1{,\ldots,}{\n}{}{}{})$ and $k=0,\ldots,n$, set:
$$ {\goth D}_{i} := \ex {i_{1}}{{\goth d}_{1}}\wedge \cdots \wedge
\ex {i_{\n}}{{\goth d}_{\n}} \quad  \text{and} \quad
{\Bbb I}_{k} := \{(\poi i1{,\ldots,}{\n}{}{}{}) \in {\Bbb N}^{\n}_{k} \; \vert \;
0\leq i_{1}\leq n_{1},\ldots,0\leq i_{\n}\leq n_{\n}\}.$$
For $l,l'$ nonnegative integers and $i$ in ${\Bbb N}^{\n}$, set:
$$ V_{l,l',i} := \ex {l}{{\goth p}_{-,\u}}\wedge {\goth D}_{i}\wedge
\ex {l'}{{\goth p}_{\u}} \quad  \text{and}$$
$$ V'_{l,l',i} := \ex {l}{{\goth p}_{-,\u}}\wedge {\goth z}\wedge {\goth D}_{i}\wedge
\ex {l'}{{\goth p}_{\u}} .$$
For $j$ in ${\Bbb N}^{\n}$ and $t=0,\ldots,\vert j \vert$, denote by $\Delta _{j}$ and
$\Delta _{j,t}$ the subsets of ${\Bbb N}^{\n}$, 
$$ \Delta _{j} := \{j'\in {\Bbb N}^{\n} \; \vert \; j'_{1}\leq j_{1},\ldots,
j'_{\n}\leq j_{\n}\}\quad  \text{and} \quad
\Delta _{j,t} := \Delta _{j}\cap {\Bbb N}^{\n}_{t}.$$
  
\begin{lemma}\label{lpt2}
Let $k=1,\ldots,n$, $(l,l')$ in ${\Bbb N}^{2}$ such that $l+l'\leq 2d$, $i$ in
${\Bbb I}_{k-l-l'}$ and $i'$ in ${\Bbb I}_{k-l-l'-1}$.

{\rm (i)} Suppose that $V_{l,l'+s,j}$ is contained in $V_{k}$ for all nonnegative
integer $s$ such that $l+l'+s\leq 2d$ and all $j$ in $\Delta _{i,\vert i \vert-s}$.
Then ${\goth d}\wedge V_{l-1,l',i}$ is contained in $V_{k}$.
   
{\rm (ii)} Suppose that $V_{l+s,l',j}$ is contained in $V_{k}$ for all nonnegative
integer $s$ such that $l+l'+s\leq 2d$ and all $j$ in $\Delta _{i,\vert i \vert-s}$.
Then ${\goth d}\wedge V_{l,l'-1,i}$ is contained in $V_{k}$.

{\rm (iii)} Suppose that $V'_{l,l'+s,j}$ is contained in $V_{k}$ for all nonnegative
integer $s$ such that $l+l'+s\leq 2d$ and all $j$ in $\Delta _{i',\vert i' \vert-s}$.
Then ${\goth d}\wedge V'_{l-1,l',i'}$ is contained in $V_{k}$.
   
{\rm (iv)} Suppose that $V'_{l+s,l',j}$ is contained in $V_{k}$ for all nonnegative
integer $s$ such that $l+l'+s\leq 2d$ and all $j$ in $\Delta _{i',\vert i' \vert-s}$.
Then ${\goth d}\wedge V'_{l,l'-1,i'}$ is contained in $V_{k}$.
\end{lemma}

\begin{proof}
For $m$ in ${\Bbb N}$ and $j$ in ${\Bbb N}^{\n}$, set:
$$ M_{m,j} := \bigoplus _{t=0}^{\vert j \vert} \bigoplus _{\iota \in \Delta _{j,t}}
{\goth D}_{\iota }\wedge \ex {m+\vert j \vert-t}{{\goth p}_{\u}}, \quad  
M_{m,j,-} := \bigoplus _{t=0}^{\vert j \vert} \bigoplus _{\iota \in \Delta _{j,t}}
{\goth D}_{\iota }\wedge \ex {m+\vert j \vert-t}{{\goth p}_{-,\u}} ,$$
$$ M'_{m,j} := \bigoplus _{t=0}^{\vert j \vert} \bigoplus _{\iota \in \Delta _{j,t}}
{\goth z}\wedge {\goth D}_{\iota }\wedge \ex {m+\vert j \vert-t}{{\goth p}_{\u}}, \quad
M'_{m,j,-} := \bigoplus _{t=0}^{\vert j \vert} \bigoplus _{\iota \in \Delta _{j,t}}
{\goth z}\wedge {\goth D}_{\iota }\wedge \ex {m+\vert j \vert-t}{{\goth p}_{-,\u}} .$$
Then $M_{m,j}$ is a $P_{\u}$-submodule of $\ex {m+\vert j \vert}{{\goth g}}$,
$M_{m,j,-}$ is a $P_{-,\u}$-submodule of $\ex {m+\vert j \vert}{{\goth g}}$, 
$M'_{m,j}$ is a $P_{\u}$-submodule of $\ex {m+\vert j \vert+1}{{\goth g}}$, 
$M'_{m,j,-}$ is a $P_{-,\u}$-submodule of $\ex {m+\vert j \vert+1}{{\goth g}}$.

(i) By hypothesis,
$$ V_{k} \supset \ex l{{\goth p}_{-,\u}}\wedge M_{l',i}\supset V_{l,l',i} .$$
Then by Proposition~\ref{pau2}(i), ${\goth d}\wedge V_{l-1,l',i}$ is contained in
$V_{k}$.

(ii) By hypothesis,
$$ V_{k} \supset \ex {l'}{{\goth p}_{\u}}\wedge M_{l,i,-}\supset V_{l,l',i} .$$
Then by Proposition~\ref{pau2}(ii), ${\goth d}\wedge V_{l,l'-1,i}$ is contained in
$V_{k}$.

(iii) By hypothesis,
$$ V_{k} \supset \ex l{{\goth p}_{-,\u}}\wedge M'_{l',i'}\supset V'_{l,l',i'} .$$
Then by Proposition~\ref{pau2}(i), ${\goth d}\wedge V'_{l-1,l',i'}$ is contained in
$V_{k}$.

(iv) By hypothesis,
$$ V_{k} \supset \ex {l'}{{\goth p}_{\u}}\wedge M'_{l',i',-}\supset V'_{l,l',i'} .$$
Then by Proposition~\ref{pau2}(ii), ${\goth d}\wedge V'_{l,l'-1,i'}$ is contained in
$V_{k}$.
\end{proof}

\begin{coro}\label{cpt2}
Let $k=1,\ldots,n$, $(l,l')$ in ${\Bbb N}^{2}$ such that $l+l'< 2d$, $i$ in
${\Bbb I}_{k-l-l'-1}$ and $i'$ in ${\Bbb I}_{k-l-l'-2}$. Then ${\goth g}\wedge V_{l,l',i}$
and ${\goth g}\wedge V'_{l,l',i'}$ are contained in $V_{k}$.
\end{coro}

\begin{proof}
Since $V_{l,l',i}$ is contained in $V_{k-1,{\goth p}}$, 
${\goth p}_{\u}\wedge V_{l,l',i}$ and ${\goth p}_{-,\u}\wedge V_{l,l',i}$ are contained
in $V_{k,{\goth p}}$. Moreover, for all nonnegative integer $s$ such that 
$l+l'+s+1\leq 2d$ and all $j$ in $\Delta _{i,\vert i \vert-s}$, 
${\goth p}_{\u}\wedge V_{l,l'+s,j}$, ${\goth p}_{-,\u}\wedge V_{l,l'+s,j}$, 
${\goth p}_{\u}\wedge V_{l+s,l',j}$, ${\goth p}_{-,\u}\wedge V_{l+s,l',j}$ are 
contained in $V_{k,{\goth p}}$. Then, by Lemma~\ref{lpt2},(i) and (ii), 
${\goth d}\wedge V_{l,l',i}$ is contained in $V_{k}$.

Since $V'_{l,l',i'}$ is contained in $V_{k-1,{\goth p}}$, 
${\goth p}_{\u}\wedge V'_{l,l',i'}$ and ${\goth p}_{-,\u}\wedge V'_{l,l',i'}$ are 
contained in $V_{k,{\goth p}}$. Moreover, for all nonnegative integer $s$ such that 
$l+l'+s+1\leq 2d$ and all $j$ in $\Delta _{i',\vert i' \vert-s}$, 
${\goth p}_{\u}\wedge V'_{l,l'+s,j}$, ${\goth p}_{-,\u}\wedge V'_{l,l'+s,j}$, 
${\goth p}_{\u}\wedge V'_{l+s,l',j}$, ${\goth p}_{-,\u}\wedge V'_{l+s,l',j}$ are 
contained in $V_{k,{\goth p}}$. Then, by Lemma~\ref{lpt2},(iii) and (iv), 
${\goth d}\wedge V'_{l,l',i'}$ is contained in $V_{k}$. By definition, 
$$ V_{k,{\goth p}} \supset {\goth p}_{-,\u}\wedge V_{k-1,{\goth p}} +
{\goth z}\wedge V_{k-1,{\goth p}} + {\goth p}_{\u}\wedge V_{k-1,{\goth p}}.$$
Hence ${\goth g}\wedge V_{l,l',i}$ and ${\goth g}\wedge V'_{l,l',i'}$ are containe
in $V_{k}$.  
\end{proof}

\begin{prop}\label{ppt2}
Let $k=1,\ldots,n$. Suppose that one of the following condition is satisfied:
\begin{itemize}
\item [{\rm (1)}] $\Pi $ is exceptional,
\item [{\rm (2)}] $\Pi $ has type ${\mathrm {D}_{\rg}}$ and $\n =3$,
\item [{\rm (3)}] $\Pi $ has classical type, $\n =2$, $2d+n_{1}$ and $2d+n_{2}$ are 
bigger than $n$.
\end{itemize}
Then $V_{k} = \ex k{{\goth g}}$.
\end{prop}

\begin{proof}
Prove the proposition by induction on $k$. For $k=1$, it is true since ${\goth g}$
is simple. Suppose $k>1$ and the proposition true for $k-1$. By Lemma~\ref{lint} and
the induction hypothesis, it is sufficient to prove that
${\goth g}\wedge V_{k-1,{\goth p}}$ is contained in $V_{k}$. As a matter of fact, we
have to prove that ${\goth g}\wedge V_{l,l',i}$ and ${\goth g}\wedge V'_{l,l',i'}$ are
contained in $V_{k}$ for $(l,l')$ in ${\Bbb N}^{2}$ such that $l+l'\leq 2d$, $i$ in
${\Bbb I}_{k-l-l'-1}$ and $i'$ in ${\Bbb I}_{k-l-l'-2}$ since
$$ V_{k-1,{\goth p}} = \bigoplus _{t=0}^{2d} \bigoplus _{(l,l')\in {\Bbb N}^{2}_{t}}
\bigoplus _{i\in {\Bbb I}_{k-t-1}} V_{l,l',i} \oplus
\bigoplus _{t=0}^{2d} \bigoplus _{(l,l')\in {\Bbb N}^{2}_{t}}
\bigoplus _{i'\in {\Bbb I}_{k-t-2}} V'_{l,l',i'} .$$

Let $(l,l')$ be in ${\Bbb N}^{2}$ such that $l+l'\leq 2d$, $i$ in ${\Bbb I}_{k-l-l'-1}$ 
and $i'$ in ${\Bbb I}_{k-l-l'-2}$. If Condition (1) or Condition (2) is satisfied, then 
$l+l'<2d$ by Proposition~\ref{prs}. As a result, by Corollary~\ref{cpt2}, 
${\goth g}\wedge V_{l,l',i}$ and ${\goth g}\wedge V'_{l,l',i'}$ are contained in $V_{k}$.
If Condition (3) is satisfied, $l+l'< 2d$ or $l+l'=2d$, $i_{1}<n_{1}$, $i_{2}<n_{2}$. In 
the first case, by Corollary~\ref{cpt2}, ${\goth g}\wedge V_{l,l',i}$ and 
${\goth g}\wedge V'_{l,l',i'}$ are contained in $V_{k}$. In the second case, 
${\goth g}\wedge V_{l,l',i}$ and ${\goth g}\wedge V'_{l,l',i'}$ are contained in 
$V_{k,{\goth p}}$, whence the proposition.
\end{proof}

\begin{rema}\label{rpt2}
By the proof of Proposition~\ref{ppt2}, when $\n=2$, 
for $k=1,\ldots,\inf\{2d+n_{1}-1,2d+n_{2}-1\}$, $V_{k}=\ex k{{\goth g}}$. 
\end{rema}

\subsection{A second particular case } \label{pt3}
In this subsection, $\vert X \vert = \rg -1$, $\Pi $ has classical type, $\n = 2$ and
$2d+n_{1}\leq n$. By Proposition~\ref{prs}, $2d+n_{2}>n$, $\rg \geq 6$ for $\Pi $ of
type ${\mathrm {A}}_{\rg}$, $\rg \geq 7$ for $\Pi $ of type ${\mathrm {B}}_{\rg}$ or
${\mathrm {C}}_{\rg}$, $\rg \geq 8$ for $\Pi $ of type ${\mathrm {D}}_{\rg}$.

For $i=(i_{0},i_{1},i_{2},i_{3},i_{4})$ in ${\Bbb N}^{5}$, set:
$$ C_{i} := \ex {i_{0}}{{\goth z}}\wedge \ex {i_{1}}{{\goth d}_{1}}\wedge
\ex {i_{2}}{{\goth d}_{2}}\wedge \ex {i_{3}}{{\goth p}_{-,\u}}\wedge
\ex {i_{4}}{{\goth p}_{\u}} .$$

Let  $k=2d+n_{1},\ldots,n$ and $j=k-2d-n_{1}$. Set:
$$ \iota := (0,n_{1},j,d,d), \quad \iota ':= (1,n_{1},j-1,d,d), $$
$$ \iota _{+}:= (0,n_{1}+1,j-1,d,d), \quad \iota '_{+}:= (1,n_{1}+1,j-2,d,d), \quad
\kappa  := (0,n_{1}+1,j,d-1,d), $$
$$  \kappa ':= (1,n_{1}+1,j-1,d-1,d), \quad  \kappa _{-}  := (0,n_{1}+1,j,d,d-1), \quad
\kappa '_{-}:= (1,n_{1}+1,j-1,d,d-1) .$$

\begin{lemma}\label{lpt3}
Denote by $M_{\iota }$ and $M_{\iota '}$ the $G$-submodules of $\ex k{{\goth g}}$
generated by $C_{\iota }$ and $C_{\iota '}$  respectively.

{\rm (i)} The subspace $M_{\iota }$ of $\ex k{{\goth g}}$ contains $C_{\kappa }$ and
$C_{\kappa _{-}}$, and $M_{\iota '}$ contains $C_{\kappa '}$ and $C'_{\kappa '_{-}}$.

{\rm (ii)} The spaces $C_{\iota _{+}}$ and $C_{\iota '_{+}}$ are contained in 
$M_{\iota }$ and $M_{\iota '}$ respectively.
\end{lemma}

\begin{proof}
(i) The subspaces of $\ex {k-d}{{\goth g}}$, 
$$ \ex {n_{1}}{{\goth d}_{1}}\wedge \ex j{{\goth d}_{2}}\wedge \ex {d}{{\goth p}_{\u}}
\quad  \text{and} \quad
\ex {n_{1}}{{\goth d}_{1}}\wedge \ex j{{\goth d}_{2}}\wedge {\goth z}\wedge
\ex {d}{{\goth p}_{\u}},$$
are invariant under $P_{\u}$. So, by Proposition~\ref{pau2}(i), $M_{\iota }$ and
$M_{\iota '}$ contain $C_{\kappa }$ and $C_{\kappa '}$ respectively. The subspaces of
$\ex {k-d}{{\goth g}}$, 
$$ \ex d{{\goth p}_{-,\u}}\wedge \ex {n_{1}}{{\goth d}_{1}}\wedge \ex j{{\goth d}_{2}}
\quad  \text{and} \quad
\ex d{{\goth p}_{-,\u}}\wedge \ex {n_{1}}{{\goth d}_{1}}\wedge \ex j{{\goth d}_{2}}
\wedge {\goth z},$$
are invariant under $P_{-\u}$. So, by Proposition~\ref{pau2}(ii), $M_{\iota }$ and
$M_{\iota '}$ contain $C_{\kappa _{-}}$ and $C_{\kappa '_{-}}$ respectively.

(ii) For $i=(i_{0},i_{1},i_{2},i_{3},i_{4})$ in ${\Bbb N}^{5}$, set:
$ i^{*} := (i_{0},i_{1},i_{2},i_{4},i_{3},) $. By corollary ~\ref{coc2}(i), for $i,j$ in
${\Bbb N}^{5}$, $C_{i}$ is orthogonal to $C_{j}$ if and only if $j\neq i^{*}$.

Denote by $M_{\iota }^{\perp}$ and  $C_{\iota }^{\perp}$ the orthogonal complements to
$M_{\iota }$ and $C_{\iota }$ in $\ex k{{\goth g}}$ respectively. By Lemma~\ref{loc1},
$M_{\iota }^{\perp}$ is the biggest $G$-module contained in $C_{\iota }^{\perp}$. Suppose
that $C_{\iota _{+}}$ is not contained in $M_{\iota }$. A contradiction is expected. 
As $\ex k{{\goth g}}$ is the direct sum of $C_{i}, \, i\in {\Bbb N}^{5}_{k}$,
$C_{\iota }^{\perp}$ is the direct sum of
$C_{i}, \, i \in {\Bbb N}^{5}_{k}\setminus \{\iota \}$ since $\iota ^{*}=\iota $.
By (i), $M_{\iota }^{\perp}$ is contained in the sum of 
$C_{i}, \, i \in {\Bbb N}^{5}_{k}\setminus \{\iota ,\kappa ,\kappa _{-}\}$. Since
$\iota _{+}^{*}=\iota _{+}$, the orthogonal complement to $C_{\iota _{+}}$ is the
sum of $C_{i}, \, i \in {\Bbb N}^{5}_{k}\setminus \{\iota _{+}\}$. Then
$M_{\iota }^{\perp}$ is not contained in the direct sum of
$C_{i}, \, i \in {\Bbb N}^{5}_{k}\setminus \{\iota ,\kappa ,\kappa _{-},\iota _{+}\}$
since $C_{\iota _{+}}$ is not contained in $M_{\iota }$. Hence for some subspace $M$
of $\ex {n_{1}+1}{{\goth d}_{1}}\wedge \ex {j-1}{{\goth d}_{2}}$,
$$ M \neq \{0\} \quad  \text{and} \quad M_{\iota }^{\perp} \supset
\ex d{{\goth p}_{-,\u}}\wedge M\wedge \ex d{{\goth p}_{\u}} $$
since $\ex d{{\goth p}_{-,\u}}\wedge \ex d{{\goth p}_{\u}}$ has dimension $1$. As a
result, by Proposition~\ref{pau2}, (i) and (ii),
$$ M_{\iota }^{\perp} \supset
\ex {d-1}{{\goth p}_{-,\u}}\wedge {\goth d}_{2}\wedge M\wedge \ex d{{\goth p}_{\u}} $$
$$ M_{\iota }^{\perp} \supset
\ex d{{\goth p}_{-,\u}}\wedge {\goth d}_{2}\wedge M\wedge \ex {d-1}{{\goth p}_{\u}} $$
since $M\wedge \ex d{{\goth p}_{\u}}$ is a $P_{\u}$-submodule of
$\ex {d+j+n_{1}}{{\goth g}}$ and $M\wedge \ex d{{\goth p}_{-,\u}}$ is
a $P_{-,\u}$-submodule of $\ex {d+j+n_{1}}{{\goth g}}$. As $j$ is smaller than
$\dim {\goth d}_{2}$ and $M$ is different from zero, ${\goth d}_{2}\wedge M \neq \{0\}$.
Then $C_{\kappa }+C_{\kappa ,-}$ is not contained in $M_{\iota }$ since
$C_{\kappa }+C_{\kappa _{-}}$ is orthogonal to $C_{i}$ for all $i$ in
${\Bbb N}^{5}_{k}\setminus \{\kappa ,\kappa _{-}\}$, whence the contradiction.

Denote by $M_{\iota '}^{\perp}$ and  $C_{\iota '}^{\perp}$ the orthogonal complements to
$M_{\iota '}$ and $C_{\iota '}$ in $\ex k{{\goth g}}$ respectively. By Lemma~\ref{loc1},
$M_{\iota '}^{\perp}$ is the biggest $G$-module contained in $C_{\iota '}^{\perp}$. 
Suppose that $C_{\iota '_{+}}$ is not contained in $M_{\iota '}$. A contradiction is 
expected. As $\ex k{{\goth g}}$ is the direct sum of $C_{i}, \, i\in {\Bbb N}^{5}_{k}$,
$C_{\iota '}^{\perp}$ is the direct sum of 
$C_{i}, \, i \in {\Bbb N}^{5}_{k}\setminus \{\iota '\}$ since ${\iota '}^{*}=\iota '$.
By (i), $M_{\iota '}^{\perp}$ is contained in the sum of 
$C_{i}, \, i \in {\Bbb N}^{5}_{k}\setminus \{\iota ',\kappa ',\kappa '_{-}\}$. Since
${\iota _{+}'}^{*}={\iota '_{+}}$, the orthogonal complement to $C_{\iota '_{+}}$ is the
sum of $C_{i}, \, i \in {\Bbb N}^{5}_{k}\setminus \{\iota '_{+}\}$. Then
$M_{\iota '}^{\perp}$ is not contained in the direct sum of
$C_{i}, \, i \in {\Bbb N}^{5}_{k}\setminus \{\iota ',\kappa ',\kappa '_{-},\iota '_{+}\}$
since $C_{\iota '_{+}}$ is not contained in $M_{\iota '}$. Hence for some subspace $M'$
of $\ex {n_{1}+1}{{\goth d}_{1}}\wedge \ex {j-2}{{\goth d}_{2}}$,
$$ M '\neq \{0\} \quad  \text{and} \quad M_{\iota '}^{\perp} \supset
\ex d{{\goth p}_{-,\u}}\wedge {\goth z}\wedge M'\wedge \ex d{{\goth p}_{\u}} $$
since ${\goth z}\wedge \ex d{{\goth p}_{-,\u}}\wedge \ex d{{\goth p}_{\u}}$ has dimension
$1$. As a result, by Proposition~\ref{pau2}, (i) and (ii),
$$ M_{\iota '}^{\perp} \supset
\ex {d-1}{{\goth p}_{-,\u}}\wedge {\goth z}\wedge {\goth d}_{2}\wedge M' 
\wedge \ex d{{\goth p}_{\u}} $$
$$ M_{\iota '}^{\perp} \supset
\ex d{{\goth p}_{-,\u}}\wedge {\goth z}\wedge {\goth d}_{2}\wedge M'
\wedge \ex {d-1}{{\goth p}_{\u}} $$
since ${\goth z}\wedge M'\wedge \ex d{{\goth p}_{\u}}$ is a $P_{\u}$-submodule of
$\ex {d+j+n_{1}}{{\goth g}}$ and ${\goth z}\wedge M'\wedge \ex d{{\goth p}_{-,\u}}$ is
a $P_{-,\u}$-submodule of $\ex {d+i+n_{1}}{{\goth g}}$. As $j$ is smaller than
$\dim {\goth d}_{2}$ and $M'$ is different from zero, ${\goth d}_{2}\wedge M' \neq \{0\}$.
Then $C_{\kappa '}+C_{\kappa ',-}$ is not contained in $M_{\iota '}$ since
$C_{\kappa '}+C_{\kappa '_{-}}$ is orthogonal to $C_{i}$ for all $i$ in
${\Bbb N}^{5}_{k}\setminus \{\kappa ',\kappa '_{-}\}$, whence the contradiction.
\end{proof}

\begin{prop}\label{ppt3}
For $k=2d+n_{1},\ldots,n$, $V_{k}$ is equal to $\ex k{{\goth g}}$.
\end{prop}

\begin{proof}
Prove the proposition by induction on $k$. Let $(l,l')$ be in ${\Bbb N}^{2}$ such
that $l+l'\leq 2d$, $i \in {\Bbb I}_{k-l-l'-1}$, $i' \in {\Bbb I}_{k-l-l'-2}$.
If $l+l'< 2d$ then ${\goth g}\wedge V_{l,l',i}$ and ${\goth g}\wedge V'_{l,l',i'}$ are
contained in $V_{k}$ by Corollary~\ref{cpt2}. If $l=l'=d$ and $i_{1}<n_{1}$ then
${\goth g}\wedge V_{l,l',i}$ is contained in $V_{k,{\goth p}}$ since $2d+n_{2}>n$ by 
Proposition~\ref{prs}. If $l=l'=d$ and $i'_{1}<n_{1}$ then ${\goth g}\wedge V'_{l,l',i'}$
is contained in $V_{k,{\goth p}}$ since $2d+n_{2}>n$ by Proposition~\ref{prs}. As a 
result, for $k=2d+n_{1}$, by Lemma~\ref{lint} and Remark~\ref{rpt2},
$V_{k}=\ex k{{\goth g}}$.

Suppose $k>2d+n_{1}$, $V_{k-1}=\ex {k-1}{{\goth g}}$, $i=(n_{1},k-2d-n_{1}-1)$ and
$i'=(n_{1},k-2d-n_{1}-2)$. By definition, $V_{k,{\goth p}}$ contains the subspaces of
$\ex k{{\goth g}}$, 
$${\goth p}_{-,\u}\wedge V_{d,d,i}, \quad {\goth p}_{\u}\wedge V_{d,d,i}, \quad
{\goth d}_{2}\wedge V_{d,d,i}, \quad {\goth p}_{-,\u}\wedge V'_{d,d,i'}, \quad 
{\goth p}_{\u}\wedge V'_{d,d,i'}, \quad {\goth d}_{2}\wedge V'_{d,d,i'} $$ 
since $i_{2}$ and $i'_{2}$ are smaller than $n_{2}$. By Lemma~\ref{lpt3}(ii),
$V_{k}$ contains ${\goth d}_{1}\wedge V_{d,d,i}$ and ${\goth d}_{1}\wedge V'_{d,d',i'}$.
Then $V_{k}$ contains ${\goth g}\wedge V_{d,d,i}$ and ${\goth g}\wedge V'_{d,d,i'}$. As a
result, by our previous remark, $V_{k}$ contains ${\goth g}\wedge V_{k-1,{\goth p}}$,
whence the proposition by Lemma~\ref{lint} and the induction hypothesis.
\end{proof}

\subsection{The general case} \label{pt4}
First, we consider the case when $X$ contains the extremities of $\Pi $.

\begin{lemma}\label{lpt4}
Suppose $\rg \geq 2$, $\n\geq 2$ and the extremities of $\Pi $ contained in $X$.
If $\vert X \vert$ is smaller than $\rg -1$ then for some $\beta $ in
$\Pi \setminus X$, $Y := \Pi \setminus \{\beta \}$ has two connected components, $X$ is
contained in $Y$ and a connected component of $Y$ is a connected component of $X$.
\end{lemma}

\begin{proof}
Suppose $\vert X \vert < \rg -1$. We consider the following cases:
\begin{itemize}
\item [{\rm (1)}] $\Pi $ has not type ${\mathrm {D}}$, ${\mathrm {E}}$,  
\item [{\rm (2)}] $\Pi $ has type ${\mathrm {D}}_{\rg}$,
\item [{\rm (3)}] $\Pi $ has type ${\mathrm {E}}_{6}$,
\item [{\rm (4)}] $\Pi $ has type ${\mathrm {E}}_{7}$,  
\item [{\rm (5)}] $\Pi $ has type ${\mathrm {E}}_{8}$.  
\end{itemize}

(1) Let $X_{1}$ be the connected component of $X$ containing $\beta _{1}$. There is only
one element $\beta $ of $\Pi \setminus X$ not orthogonal to $X_{1}$. Then
$Y:=\Pi \setminus \{\beta \}$ has two connected components and $X_{1}$ is a connected
component of $Y$.

(2) Let $X_{1}$ be the connected component of $X$ containing $\beta _{1}$.  As
$\beta _{\rg}$ and $\beta _{\rg -1}$ are in $X$, for some $i$ smaller than $\rg -2$,
$\beta _{i}$ is not in $X$ since $\vert X \vert < \rg -1$. Then there is only one element
$\beta $ in $\Pi \setminus X$ not orthogonal to $X_{1}$ so that 
$Y:=\Pi \setminus \{\beta \}$ has two connected components and $X_{1}$ is a connected 
component of $Y$.

(3) As $\beta _{1}$, $\beta _{2}$, $\beta _{6}$ are in $X$, $\beta _{3}$ or
$\beta _{5}$ is not in $X$ since $\vert X \vert < \rg -1$. Setting 
$Y_{i}:=\Pi \setminus \{\beta _{i}\}$ for $i=3,5$, $Y_{i}$ has two connected components 
and for some $i$, $X$ is contained in $Y_{i}$ and a connected component of $X$ is a 
connected component of $Y_{i}$.

(4) As $\beta _{1}$, $\beta _{2}$, $\beta _{7}$ are in $X$, $\beta _{3}$ or
$\beta _{5}$ or $\beta _{6}$ is not in $X$ since $\vert X \vert < \rg -1$. Setting 
$Y_{i}:=\Pi \setminus \{\beta _{i}\}$ for $i=3,5,6$, $Y_{i}$ has two connected components
and for some $i$, $X$ is contained in $Y_{i}$ and a connected component of $X$ is a 
connected component of $Y_{i}$.

(5) As $\beta _{1}$, $\beta _{2}$, $\beta _{8}$ are in $X$, $\beta _{3}$ or
$\beta _{5}$ or $\beta _{6}$ or $\beta _{7}$ is not in $X$ since 
$\vert X \vert < \rg -1$. Setting $Y_{i}:=\Pi \setminus \{\beta _{i}\}$ for $i=3,5,6,7$, 
$Y_{i}$ has two connected components and for some $i$, $X$ is contained in $Y_{i}$ and a 
connected component of $X$ is a connected component of $Y_{i}$.
\end{proof}

\begin{prop}\label{ppt4}
Let $k=1,\ldots,n$. Suppose that Theorem~{\rm \ref{tint}} is true for the simple
algebras of rank smaller than $\rg$ and $X$ contains the extremities of $\Pi $. Then 
$V_{k}=\ex k{{\goth g}}$.
\end{prop}

\begin{proof}
  As $X$ contains the extremities of $\Pi $ and is different from $\Pi $,
  $\rg \geq 3$ and  $\n\geq 2$. By Proposition~\ref{ppt2}, Remark~\ref{rpt2} and
  Proposition~\ref{ppt3}, $V_{k}=\ex k{{\goth g}}$ when $\vert X \vert = \rg -1$. In
  particular, $V_{k}=\ex k{{\goth g}}$ when $\rg =3$. Suppose $\rg > 3$ and
  $\vert X \vert < \rg -1$.

  Let $Y$ be as in Lemma~\ref{lpt4}. Then ${\goth d}_{Y}$ has two simple factors
  ${\goth d}_{1}$ and ${\goth a}$ and ${\goth d}_{1}$ is a simple factor of ${\goth d}$.
  Denote by $V_{k,Y}$ the $G$-submodule of $\ex k{{\goth g}}$ generated by
  $V_{k,{\goth p}_{Y}}$. Then, by Proposition~\ref{ppt2}, Remark~\ref{rpt2} and
  Proposition~\ref{ppt3}, $V_{k,Y}=\ex k{{\goth g}}$. The intersection
  ${\goth q} := {\goth a}\cap {\goth p}$ is a parabolic subalgebra of ${\goth a}$. Let
  $E'$ be the intersection of $E$ and ${\goth a}$. Then $E$ is the direct sum of $E'$
  and $E_{Y}$. As a result, setting  $n_{*} := \b a{}-\j a{}$,
$$ V_{k,{\goth p}} = \bigoplus _{i=0}^{n_{1}} \bigoplus _{j=0}^{n_{*}}
\ex i{{\goth d}_{1}}\wedge V_{j,{\goth q}}\wedge \ex {k-i-j}{E_{Y}} .$$
Let $A$ be the connected closed subgroup of $G$ whose Lie algebra is ${\goth a}$.
By the hypothesis, for $j=1,\ldots,n_{*}$, the $A$-submodule of 
$\ex j{{\goth a}}$, generated by $V_{j,{\goth q}}$, is equal to 
$\ex j{{\goth a}}$. Hence, by Lemma~\ref{lint}, $V_{k,{\goth p}_{Y}}$ is contained in
$V_{k}$ since ${\goth d}_{1}$ and $E_{Y}$ are invariant under $A$, whence
$V_{k}=\ex k{{\goth g}}$. 
\end{proof}

To finish the proof of Theorem~\ref{tint}, we have to consider the case when $X$ does
not contain all the extremities of $\Pi $.

\begin{lemma}\label{l2pt4}
Suppose that $X$ does not contain all the extremities of $\Pi $.

{\rm (i)} There exists a sequence 
$$ \poi X0{\subset \cdots \subset}{m}{}{}{} = \Pi $$
of connected subsets of $\Pi $ satisfying the following conditions:
\begin{itemize}
\item [{\rm (1)}] for $i=1,\ldots,m$, $\vert X_{i}\setminus X_{i-1} \vert=1$, 
\item [{\rm (2)}] $X$ contains the extremities of $X_{0}$.
\end{itemize}

{\rm (ii)} For $i=0,\ldots,m$, let ${\goth a}_{i}$ be the subalgebra of ${\goth g}$
generated by ${\goth g}^{\pm \beta }, \, \beta \in X_{i}$. Then ${\goth a}_{i}$ is a simple
algebra and ${\goth p}_{i}:={\goth p}\cap {\goth a}_{i}$ is a parabolic subalgebra of 
${\goth a}_{i}$.

{\rm (iii)} For $i=0,\ldots,m$, $E$ is the direct sum of $E_{i}:=E\cap {\goth a}_{i}$
and $E_{X_{i}}$.
\end{lemma}

\begin{proof}
(i) Define $X_{i}$ by induction on $i$. Let $X_{0}$ be a connected subset of $\Pi $, 
containing $X$ of minimal cardinality. By minimality of $\vert X_{0} \vert$, $X$ contains
the extremities of $X_{0}$. Suppose $i>0$ and $X_{i-1}$ defined. If $X_{i-1}=\Pi $
there is nothing to do. Suppose $X_{i-1}\neq \Pi $. As $\Pi $ is connected, there is 
some $\beta $ in $\Pi \setminus X_{i-1}$, not orthogonal to an extremity of $X_{i-1}$. 
Then $X_{i} := X_{i-1}\cup \{\beta \}$ is a connected subset of $\Pi $ since so is
$X_{i-1}$, whence the assertion.

(ii) As $X_{i}$ is connected, ${\goth a}_{i}$ is a simple algebra. For $\alpha $ in 
$<X_{i}>$, ${\goth g}^{\alpha }$ is contained in ${\goth a}_{i}\cap {\goth p}$. Hence 
${\goth p}_{i}$ contains the Borel subalgebra of ${\goth a}_{i}$ generated by 
${\goth h}\cap {\goth a}_{i}$ and ${\goth g}^{\beta }, \, \beta \in X_{i}$, whence the
assertion.

(iii) Let $\alpha $ be a positive root such that ${\goth g}^{\alpha }$ is contained in 
$E$. If $\alpha $ is in $<X_{i}>$ then ${\goth g}^{\alpha }$ and ${\goth g}^{-\alpha }$
are contained in $E_{i}$. Otherwise, ${\goth g}^{\alpha }$ and ${\goth g}^{-\alpha }$
are contained in $E_{X_{i}}$ by definition. 

Let $z$ be in $E\cap {\goth h}$. By definition, ${\goth h}\cap E_{X_{i}}$ is the 
orthogonal complement to ${\goth a}_{i}$ in ${\goth h}$. Then $z=z_{1}+z_{2}$ with 
$z_{1}$ in ${\goth a}_{i}\cap {\goth h}$ and $z_{2}$ in $E_{X_{i}}$. Hence $z_{1}$ is 
orthogonal to ${\goth a}_{i}\cap {\goth p}$. As a result, $z_{1}$ is in $E_{i}$ and 
$z$ is in $E_{i}+E_{X_{i}}$, whence the assertion.
\end{proof}

We can now give the proof of Theorem~\ref{tint}.

\begin{proof}
Prove the theorem by induction on $\rg$. First of all, for $X$ empty subset of $\Pi $,
$V_{k,{\goth p}_{X}}=\ex k{{\goth g}}$. For $\rg =1$, $n=1$. Hence the theorem is true in
this case and we can suppose $X$ nonempty and $\rg \geq 2$. By Proposition~\ref{ppt1}
and Remark~\ref{rpt1}, $V_{k}=\ex k{{\goth g}}$ when $X$ is connected. In particular,
the theorem is true in rank $2$. Then, by Proposition~\ref{ppt4}, the theorem is true
for $\rg =3$ since in this case $X$ contains all the extremities of $\Pi $ when it is
not connected.

Suppose $\rg > 3$ and the theorem true for the simple algebras of rank smaller than
$\rg$. By induction hypothesis and Proposition~\ref{ppt4}, $V_{k}=\ex k{{\goth g}}$
when $X$ contains all the extremities of $\Pi $. So, we can suppose that $X$ does not
contain all the extremities of $\Pi $. Let  $\poi X0{,\ldots,}{m}{}{}{}$ be as in
Lemma~\ref{l2pt4}. For $i=0,\ldots,m$, set $e_{i} := \vert <X_{i}> \vert$ and prove by
induction on $i$ the inclusion
$$ \bigoplus _{j=0}^{e_{i}} \ex j{{\goth a}_{i}}\wedge  \ex {k-j}{E_{X_{i}}} \subset 
V_{k}.$$

For $i=0,\ldots,n$, denote by $A_{i}$ the connected closed subgroup of $G$ whose Lie
algebra is ${\goth a}_{i}$. By Lemma~\ref{l2pt4},(ii) and (iii), 
$$ V_{k,{\goth p}} = \bigoplus _{j=0}^{e_{i}} V_{j,{\goth p}_{i}}\wedge 
\ex {k-j}{E_{X_{i}}},$$
for $i=0,\ldots,m$. Then, by Proposition~\ref{ppt4}, the induction hypothesis and
Lemma~\ref{lint}, the inclusion is true for $i=0$ since $E_{X_{0}}$ is invariant under
$A_{0}$. Suppose $i>0$ and the inclusion true for $i-1$. Let $E'_{X_{i}}$ be the
intersection of $E_{X_{i-1}}$ and ${\goth a}_{i}$. Denote by ${\goth q}_{i}$ the parabolic
subalgebra of ${\goth a}_{i}$ containing ${\goth b}\cap {\goth a}_{i}$ and such that
${\goth a}_{i-1}$ is the derived algebra of the reductive factor of ${\goth q}_{i}$
containing ${\goth h}\cap {\goth a}_{i}$. Then
$$ E_{X_{i-1}} = E'_{X_{i}} \oplus E_{X_{i}}, \quad  
{\goth a}_{i} = {\goth a}_{i-1} \oplus E'_{X_{i}} ,$$
$$ \ex j{{\goth a}_{i-1}}\wedge  \ex {k-j}{E_{X_{i-1}}} =   \bigoplus _{l=0}^{k-j}  
\ex j{{\goth a}_{i-1}}\wedge \ex l{E'_{X_{i}}} \wedge \ex {k-j-l}{E_{X_{i-1}}}$$
for $j=0,\ldots,e_{i-1}$. As a result,
$$ \bigoplus _{j=0}^{e_{i-1}} \ex j{{\goth a}_{i-1}}\wedge  \ex {k-j}{E_{X_{i-1}}} =
\bigoplus _{j=0}^{e_{i}} V_{j,{\goth q}_{i}}\wedge \ex {k-j}{E_{X_{i}}} .$$
By Proposition~\ref{ppt1} and Remark~\ref{rpt1}, for $j=0,\ldots,e_{i}$, the
$A_{i}$-submodule of $\ex j{{\goth a}_{i}}$ generated by $V_{j,{\goth q}_{i}}$ is equal 
to $\ex j{{\goth a}_{i}}$ since ${\goth a}_{i-1}$ is simple
and $\j ai - \j a{i-1}=1$. Then, by Lemma~\ref{lint}, the $A_{i}$-submodule of
$\ex k{{\goth g}}$ generated by
$$ \bigoplus _{j=0}^{e_{i-1}} \ex j{{\goth a}_{i-1}}\wedge  \ex {k-j}{E_{X_{i-1}}}$$
is equal to
$$ \bigoplus _{j=0}^{e_{i}} \ex j{{\goth a}_{i}}\wedge  \ex {k-j}{E_{X_{i}}}$$
since $E_{X_{i}}$ is invariant under $A_{i}$, whence the assertion and the theorem since
for $i=m$ the sum is equal to $\ex k{{\goth g}}$.
\end{proof}

\appendix

\section{Some remarks on root systems} \label{rs}
Let $\beta $ be in $\Pi $ and $X := \Pi \setminus \{\beta \}$. Set
${\goth p}_{\u} := {\goth p}_{\u,X}$ and $d := \dim {\goth p}_{\u,X}$. The goal of
the section is the following proposition:

\begin{prop}\label{prs}
{\rm (i)} Suppose $\Pi $ of type ${\mathrm {A}}_{\rg}$ and $X$ not connected. Then
$\beta = \beta _{s+1}$ for some $s$ in $\{1,\ldots,\rg -2\}$,
$$\n =2, \quad n_{1} = \frac{s(s+1)}{2}, \quad n_{2} = \frac{(\rg-s-1)(\rg-s)}{2} .$$
Moreover, if $2d + n_{1}\leq n$ then 
$$ \rg \geq 6, \quad
s \leq \frac{1}{6}(2\rg -3 - \sqrt{ 4\rg ^{2} + 12 \rg + 9}), \quad   2d+n_{2} > n .$$

{\rm (ii)} Suppose $\Pi $ of type ${\mathrm {B}}_{\rg}$ or ${\mathrm {C}}_{\rg}$ and $X$
not connected. Then $\beta = \beta _{s+1}$ for some $s$ in $\{1,\ldots,\rg -2\}$,
$$\n =2, \quad n_{1} = \frac{s(s+1)}{2}, \quad n_{2} = (\rg-s-1)^{2} .$$
Moreover, if $2d + n_{1}\leq n$ then 
$$ \rg \geq 7, \quad
s \leq \frac{1}{10}(8\rg -9 - \sqrt{ 24\rg ^{2} + 16 \rg + 1}),  \quad 2d+n_{2} > n .$$

{\rm (iii)} Suppose $\Pi $ of type ${\mathrm {D}}_{\rg}$. If $\beta =\beta _{\rg -2}$ then
$2d > n$. If $\beta $ is different from $\beta _{\rg -2}$ and $X$ is not connected, then
$$\n =2, \quad n_{1} = \frac{s(s+1)}{2}, \quad n_{2} = (\rg-s-1)^{2} .$$
Moreover, if $2d + n_{1}\leq n$ then 
$$ \rg \geq 8, \quad
s \leq \frac{1}{10}(8\rg -13 - \sqrt{ 24\rg ^{2} - 8\rg + 9}),  \quad 2d+n_{2} > n .$$

{\rm (iv)} Suppose that $\Pi $ is exceptional. If $2d\leq n$ then $X$ is connected.
\end{prop}

We prove the proposition case by case. So, in the classical case, we suppose
$\rg \geq 3$ and $X$ not connected.

\subsection{Type ${\mathrm {A}_{\rg}}$} \label{rs1}
As $X$ is not connected, $\n=2$ and $\beta = \beta _{s+1}$ for some $s$ in
$\{1,\ldots,\rg-2\}$. Then
$$ n_{1} = \frac{s(s+1)}{2}, \quad n_{2} = \frac{(\rg-s-1)(\rg-s)}{2}, \quad
d = n -n_{1}-n_{2},$$
$$ n - 2d - n_{1} = \frac{1}{2}(3s^2  + (-4\rg + 3)s + \rg^2- 3\rg).$$
If $n-2d-n_{1} \geq 0$ then
$$ s \leq \frac{1}{6}(4\rg -3 - \sqrt{4\rg^{2} + 12\rg + 9}) \quad  \text{or} \quad
s \geq \frac{1}{6}(4\rg -3 + \sqrt{4\rg^{2} + 12\rg + 9}) .$$
As $s\geq 1$, the first inequality is possible only if $\rg \geq 6$. The second
inequality is impossible since its right hand side is bigger than $\rg -2$ and
$s$ is at most $\rg -2$.

By the above equalities,
$$ n -2d - n_{2} = 3s^2  + (-2\rg + 3)s - 2\rg.$$
If the left hand side is nonnegative then
$$ s \leq \frac{1}{6}(2\rg -3 - \sqrt{4\rg^{2} + 12\rg + 9}) \quad  \text{or} \quad
s \geq \frac{1}{6}(2\rg -3 + \sqrt{4\rg^{2} + 12\rg + 9}) .$$
The first inequality is impossible since its right hand side is negative. The second
inequality is possible only if $\rg \geq 7$ since $s\leq \rg -2$. Moreover, it is not
possible to have $n\geq 2d+n_{1}$ and $n\geq 2d+n_{2}$ since
$$ \frac{1}{6}(2\rg -3 + \sqrt{4\rg^{2} + 12\rg + 9}) >
\frac{1}{6}(4\rg -3 - \sqrt{4\rg^{2} + 12\rg + 9}),$$
whence Assertion (i) of Proposition~\ref{prs}.

\subsection{Type ${\mathrm {B}_{\rg}}$ or ${\mathrm {C}_{\rg}}$}\label{rs2}
As $X$ is not connected, $\n=2$ and $\beta = \beta _{s+1}$ for some $s$ in
$\{1,\ldots,\rg-2\}$. Then
$$ n_{1} = \frac{s(s+1)}{2}, \quad n_{2} = (\rg-s-1)^{2}, \quad
d = n -n_{1}-n_{2},$$
$$ n - 2d - n_{1} = \frac{1}{2}(5s^2  + (-8\rg + 9)s + 2\rg^2- 8\rg + 4).$$
If $n-2d-n_{1} \geq 0$ then
$$ s \leq \frac{1}{10}(8\rg -9 - \sqrt{24\rg^{2} + 16\rg + 1}) \quad  \text{or} \quad
s \geq \frac{1}{10}(8\rg -9 + \sqrt{24\rg^{2} + 16\rg + 1}) .$$
As $s\geq 1$, the first inequality is possible only if $\rg \geq 7$. The second
inequality is impossible since its right hand side is bigger than $\rg -2$ and
$s$ is at most $\rg -2$.

By the above equalities,
$$ n -2d - n_{2} = 2s^2  + (-2\rg + 5)s - 4\rg+4.$$
If the left hand side is nonnegative then
$$ s \leq \frac{1}{4}(2\rg -5 - \sqrt{4\rg^{2} + 12\rg - 7}) \quad  \text{or} \quad
s \geq \frac{1}{4}(2\rg -5 + \sqrt{4\rg^{2} + 12\rg - 9}) .$$
The first inequality is impossible since its right hand side is negative. The second
inequality is impossible since
$$s\leq \rg -2 \quad  \text{and} \quad
\frac{1}{4}(2\rg -5 + \sqrt{4\rg^{2} + 12\rg - 9}) > \rg -2,$$
whence Assertion (ii) of Proposition~\ref{prs}.

\subsection{Type ${\mathrm {D}_{\rg}}$} \label{rs3}
As $X$ is not connected, $\beta $ is different from $\beta _{1}$, $\beta _{\rg-1}$,
$\beta _{\rg}$. If $\beta = \beta _{\rg -2}$ then $X$ has three connected components and
$$ d = \rg(\rg -1) - 2 - \frac{1}{2}(\rg-3)(\rg -2) .$$
In this case $n<2d$. Suppose $\rg \geq 5$ and $\beta = \beta _{s+1}$ for some $s$ in
$\{1,\ldots,\rg-4\}$. Then
$$ n_{1} = \frac{s(s+1)}{2}, \quad n_{2} = (\rg-s-1)(\rg -s-2), \quad
d = n-n_{1}-n_{2},$$
$$ n-2d - n_{1} = 5 s^{2} - s(4\rg -7) + \rg^{2}-5\rg+4 .$$
If $n-2d-n_{1}\geq 0$ then
$$ s \leq \frac{1}{10}(8\rg -13 - \sqrt{24\rg^{2}- 8\rg+9}) \quad  \text{or} \quad
s \geq \frac{1}{10}(8\rg -13 + \sqrt{24\rg^{2}-8\rg+9}) .$$
As $s\geq 1$, the first inequality is possible only if $\rg \geq 8$. The second
inequality is impossible since its right hand side is bigger than $\rg -4$ and
$s$ is at most $\rg -4$.

By the above equalities,
$$ n -2d - n_{2} = 2s^2  + (-2\rg + 4)s - 2\rg+2.$$
If the left hand side is nonnegative then
$$ s \leq -1 \quad  \text{or} \quad s\geq \rg-1 .$$
These inequalities are impossible since $s$ is positive and smaller than $\rg -3$,
whence Assertion (iii) of Proposition~\ref{prs}.

\subsection{The exceptional case.} \label{rs4}
Set ${\goth l}:={\goth l}_{X}$, ${\goth d} := {\goth d}_{X}$. Then
$2d = \dim {\goth g}-\dim {\goth l}$. For each case, we give all the possible dimensions
of ${\goth l}$ when $\vert X \vert = \rg -1$.

(a) The algebra ${\goth g}$ has type ${\mathrm {G}}_{2}$. Then $X$ is connected, whence
Assertion (iv) of Proposition~\ref{prs} for this case.

(b) The algebra ${\goth g}$ has type ${\mathrm {F}}_{4}$. In this case $n=24$ and
$$ \dim {\goth l} \in \{12,22\} \quad  \text{whence} \quad 2d \in \{40,30\}$$
and Assertion (iv) of Proposition~\ref{prs} for this case.

(c) The algebra ${\goth g}$ has type ${\mathrm {E}}_{6}$. In this case $n=36$ and
$$ \dim {\goth l} \in \{20,28,36,46\} \quad  \text{whence} \quad
2d \in \{58,50,42,32 \}$$
and Assertion (iv) of Proposition~\ref{prs} for this case since ${\goth d}$ is simple of
type ${\mathrm {D}}_{5}$ when $2d = 32$.

(d) The algebra ${\goth g}$ has type ${\mathrm {E}}_{7}$. In this case $n=63$ and
$$ \dim {\goth l} \in \{27,33,39,49,67,79\} \quad  \text{whence} \quad
2d \in \{106,100,94,84,66,54\}$$
and Assertion (iv) of Proposition~\ref{prs} for this case since ${\goth d}$ is simple of
type ${\mathrm {E}}_{6}$ when $2d=54$.

(e) The algebra ${\goth g}$ has type ${\mathrm {E}}_{8}$. In this case $n=120$ and
$$ \dim {\goth l} \in \{36,40,52,54,64,82,92,134\} \quad  \text{whence} \quad
2d \in \{212,208,196,194,184,166,156,114\}$$
and Assertion (iv) of Proposition~\ref{prs} for this case since ${\goth d}$ is simple of
type ${\mathrm {E}}_{7}$ when $2d=114$.

\end{document}